\documentclass[paper=a4,fontsize=12pt, DIV=12, BCOR=0.00mm, twoside=false, english, reqno, abstracton]{scrartcl}

\usepackage{babel}				

\usepackage[utf8]{inputenc}	
\usepackage{lmodern} 			
\usepackage[T1]{fontenc}		
\usepackage[final]{microtype}
\setkomafont{sectioning}{\sffamily\mdseries} 
\usepackage{relsize} 
\usepackage{graphicx}	
\makeatletter
\newcommand{\ifKOMAClass}[1]{
\@ifundefined{KOMAClassName}{#1}
}
 \makeatother

 \addtokomafont{caption}{\small}

\usepackage{setspace} 

\usepackage{graphicx}			
\usepackage{float} 				

\usepackage{enumitem} 
\usepackage{mathrsfs}                            
\usepackage{makeidx}                             

\usepackage{amsmath,amsfonts,amssymb,amsxtra, accents} 

\usepackage[amsmath,thmmarks]{ntheorem} 
\usepackage{tabularx}

\usepackage{dsfont}
\usepackage{mathtools}

\theoremstyle{break}
\theoremheaderfont{\normalfont\scshape}
\theorembodyfont{\normalfont\itshape}
\theoremseparator{:\\}
\theoremnumbering{arabic}
\theoremsymbol{}

\newtheorem{satz}{Satz}[section]		
\newtheorem{defi}[satz]{Definition}  
\newtheorem{lem}[satz]{Lemma}			
\newtheorem{lemma}[satz]{Lemma}
\newtheorem{theorem}[satz]{Theorem}	
\newtheorem*{theorem-non}{Theorem}
\newtheorem{cor}[satz]{Corollary}

\newtheorem{remark}[satz]{Remark}

\newtheorem*{ass}{Assumptions}
\newtheorem*{ass2}{Assumption}

\theoremstyle{nonumberplain}
\theoremheaderfont{\normalfont\mdseries\scshape}
\theorembodyfont{\normalfont\upshape}
\theoremseparator{:\\}
\theoremnumbering{arabic}
\theoremsymbol{\rule{1ex}{1ex}}

\newtheorem{proof}{Proof}
\usepackage{scrlayer-scrpage}
\clearscrheadfoot
\ohead{\rightmark}
\ifoot{}
\ofoot[p. \pagemark ]{p. \pagemark }

\automark[]{section} 

\deftriplepagestyle{myStyle1}{ }{ }{ }{}{ }{Seite \pagemark  }
\deftriplepagestyle{myStyle2}{ }{ }{ }{}{\pagemark }{ }

\KOMAoptions{
	numbers=noenddot, 
	toc=indented,%
	bibliography=openstyle,
      listof=left, 
   listof=indented, 
      listof=totoc, 
     bibliography=openstyle,%
     bibliography=totoc, 
	}

\usepackage[%
backend=biber,
useprefix=true,
style=alphabetic,          
sorting=nyvt,  				
minnames=1,
maxnames=5,
url=false,
doi=false,
isbn=false,
giveninits=true,  
]{biblatex}
\usepackage{mathtools}
\DeclarePairedDelimiter\ceil{\lceil}{\rceil}

\newcommand{\NE}[1]{\abs{ #1}}   %
\newcommand{\given}{\;\middle\vert\;}

\newcommand{\oB}[3]{{B}_{#3}^{}\left(#1,#2\right)}
\newcommand{\cB}[3]{\overline{{B}}_{#3}^{}\left(#1,#2\right)}

\DeclareMathOperator*{\crr}{\ensuremath{r_0}	}
\DeclareMathOperator*{\cs}{\ensuremath{s_0}	}
\DeclareMathOperator{\grad}{\ensuremath{deg}	}
\DeclareMathOperator{\vol}{\ensuremath{vol}	}
\newcommand{\RZ}{\mathbb{R}}
\newcommand{\GZ}{\mathbb{Z}}
\newcommand{\NZ}{\mathbb{N}}

\newcommand{\pr}[1]{\ensuremath{\xi_{#1}}}
\newcommand{\dista}{\ensuremath{ \rho}} %
\newcommand{\distb}[1]{\ensuremath{ \rho_{#1}}} %
\newcommand{\dist}[2]{\ensuremath{ \rho_{#1}\left(#2 \right)}} %
\newcommand{\EL}[1] { \ensuremath{EL \left(#1\right)}} %
\newcommand{\prob}[1]{\ensuremath{\mathbf{P}\left(#1\right)}}					%

\newcommand{\ew}[1]{\ensuremath{\mathbf{E}\left(#1\right)}}						%
\newcommand{\ewt}[2][\theta]{\ensuremath{\mathbf{E}_{#1}\left(#2\right)}}		%
\newcommand{\ewx}[2]{\ensuremath{\mathbf{E}_{#1}\left(#2\right)}}  			%
\newcommand{\BinV}{\mathrm{Bin}} 	%
\newcommand{\manifold}[1]{\mathcal{#1}}

\newcommand{\diam}[1]{\mathrm{diam}\left(#1\right)}
\newcommand{\secK}{\mathrm{sec}}
\DeclareMathOperator*{\argmin}{argmin}

\providecommand{\norm}[1]{\left\lVert #1 \right\rVert}
\providecommand{\abs}[1]{\left\lvert #1 \right\rvert}
\newcommand{\bigO}{\ensuremath{\mathcal{O}}} %

\newcommand{\Indik}{\mathds{1}}			%
\newcommand{\Einh}{\mathds{I}}			%

\usepackage[intoc, prefix]{nomencl}
\usepackage{ifthen}

\renewcommand{\nomgroup}[1]{
	\renewcommand{\makelabel}[1][]{##1}
	\item[~]
	\ifthenelse{\equal{#1}{A}}{%
		\item[\textbf{Allgemeine Abkürzungen}]}{%
		\ifthenelse{\equal{#1}{B}}{%
			\item[\textbf{Mathematische Symbole 1}]}{%
			\ifthenelse{\equal{#1}{C}}{%
				\item[\textbf{Mathematische Symbole 2}]}{%
				\ifthenelse{\equal{#1}{D}}{%
					\item[\textbf{Mathematische Symbole 3}]}{%
					\ifthenelse{\equal{#1}{E}}{%
						\item[\textbf{Mathematische Symbole 4}]}{%
	}}}}}%
	\item[~]
	\let\makelabel\nomlabel
}
\makenomenclature
\usepackage{xargs}  
\usepackage{tikz}
\PassOptionsToPackage{dvipsnames}{xcolor}

\usepackage{hyperref}
\usepackage[all]{hypcap} 
\usepackage[capitalise,nameinlink, nosort]{cleveref} 
\crefname{equation}{}{}  
\crefname{ass}{assumption}{assumptions}
\crefname{cor}{Corollary}{Corollaries}
\crefname{lem}{Lemma}{Lemmata}
\crefname{proof}{Proof}{Proofs}
\crefname{defi}{Definition}{Definitions}
\crefname{enumi}{}{}
\usepackage{nameref}

\makeatletter
\def\myEnumCounter#1{\expandafter\@myEnumCounter\csname c@#1\endcsname}
\def\@myEnumCounter#1{\ifcase#1\or A1\or A2\or A3\or S1\or S1*\or S2\fi}
\makeatother
\makeatletter
\def\myEnumCounterG#1{\expandafter\@myEnumCounterG\csname c@#1\endcsname}
\def\@myEnumCounterG#1{\ifcase#1\or G2\or G3\or G4\fi}
\makeatother
\makeatletter
\def\myEnumCounterGG#1{\expandafter\@myEnumCounterGG\csname c@#1\endcsname}
\def\@myEnumCounterGG#1{\ifcase#1\or G1\fi}
\makeatother
\makeatletter
\def\myEnumCounterS#1{\expandafter\@myEnumCounterS\csname c@#1\endcsname}
\def\@myEnumCounterS#1{\ifcase#1\or S2*\fi}
\makeatother
\makeatletter
\def\myEnumCounterE#1{\expandafter\@myEnumCounterE\csname c@#1\endcsname}
\def\@myEnumCounterE#1{\ifcase#1\or E1\fi}
\makeatother
\AddEnumerateCounter{\myEnumCounter}{\@myEnumCounter}{S1* }
\AddEnumerateCounter{\myEnumCounterG}{\@myEnumCounterG}{G2 }
\AddEnumerateCounter{\myEnumCounterGG}{\@myEnumCounterGG}{G1 }
\AddEnumerateCounter{\myEnumCounterS}{\@myEnumCounterS}{S2* }
\AddEnumerateCounter{\myEnumCounterE}{\@myEnumCounterE}{E1 }
 \newlist{assenum}{enumerate}{1} 
\setlist[assenum]{label=\textbf{\myEnumCounter*},ref=\myEnumCounter*} 
 \newlist{assenumG}{enumerate}{1} 
\setlist[assenumG]{label=\textbf{\myEnumCounterG*},ref=\myEnumCounterG*}
 \newlist{assenumGG}{enumerate}{1} 
\setlist[assenumGG]{label=\textbf{\myEnumCounterGG*},ref=\myEnumCounterGG*}
 \newlist{assenumS}{enumerate}{1} 
\setlist[assenumS]{label=\textbf{\myEnumCounterS*},ref=\myEnumCounterS*}
 \newlist{assenumE}{enumerate}{1} 
\setlist[assenumE]{label=\textbf{\myEnumCounterE*},ref=\myEnumCounterE*}
\crefalias{assenumi}{ass}
\crefalias{assenumGi}{ass}
\crefalias{assenumGGi}{ass}
\crefalias{assenumSi}{ass}
\crefalias{assenumEi}{ass}
\usepackage{authblk}

\begin{document}
\title{Volume Doubling Condition and a Local Poincar\'{e} Inequality on Unweighted Random Geometric Graphs}

\author{Franziska G\"obel}
\author{Gilles Blanchard}
\affil{ Institute of Mathematics\\ University of Potsdam, Germany }
\affil{\large\textit {goebel@uni-potsdam.de}\\  \textit{gilles.blanchard@math.uni-potsdam.de}}

\date{July 5, 2019}

\maketitle


\begin{abstract}
	The aim of this paper is to establish two fundamental measure-metric properties of particular random geometric graphs. We consider $\varepsilon$-neighborhood graphs whose vertices are drawn independently and identically distributed from a common distribution defined on a regular submanifold of $\RZ^K$. 
	We show that a volume doubling condition (VD) and local Poincaré inequality (LPI) hold for the random geometric graph
	(with high probability, and uniformly over all shortest path distance balls in a certain radius range)
	under suitable regularity conditions of the underlying submanifold  and the sampling distribution.
	\end{abstract}

        \section*{Acknowledgment}
        	Both authors acknowledge support for this project
	from the German Research Foundation (DFG research group FOR1735 - Structural Inference in Statistics: Adaptation and Efficiency).
        The first author acknowledges mobility support from the UFA-DFH through the French-German Doktorandenkolleg CDFA 01-18.
\section{Introduction}

The aim of this paper is to establish fundamental measure-metric properties of particular random geometric graphs.
The motivation for this study comes from high-dimensional data analysis. Informally speaking, assuming a data sample $(X_1,\ldots,X_n)$
taking values in $\mathbb{R}^K$ is drawn independently and identically distributed from a common distribution $\mathbb{P}$, if the ambient dimensionality $K$ is large, most statistical estimation problems (for example, regression
where each point $X_i$ comes with an associated real-valued random label $Y_i$ and the goal is to estimate the function
$f(x) = E[Y|X=x]$) will suffer from the ``curse of dimensionality'', that is, convergence rates (as a function of $n$) to the estimation target will be hopelessly slow in a minimax sense (in the regression example, this is the case if
$\mathbb{P}$ is comparable to Lebesgue on an open set of $\mathbb{R}^K$, and for any given classical smoothness class of the target function.) To alleviate the high dimensionality issue, additional structural assumptions on the distribution
have to be made to reduce the inherent complexity of the problem. Such structural assumptions can be of 
very different nature; in the present work, we focus on the often considered assumption that the support
of $\mathbb{P}$ is a regular submanifold $\mathcal{M}$ of $\mathbb{R}^K$ (of dimension $k \ll K$).

In such a setting, a central issue is that the supporting submanifold $\mathcal{M}$ is unknown. A fundamental tool introduced to recover, implicitly or explicitly, this unknown structure is to construct a neighborhood graph based
on the observed sample, that is, a geometrical graph whose vertices are the sample points themselves, and edges join
neighbor points, defined in a suitable sense based on the ambient Euclidean distance in $\mathbb{R}^K$.
In this work we consider the simple case of $\varepsilon$-neighborhood graphs, where neighbor points are those whose
ambient Euclidean distance to each other is less than $\varepsilon$ (a fixed in advance constant). Such graphs have been considered for data analysis
purposes since the seminal works \cite{Tenenbaum2000,TechreportIsomap,Belkin2003}.  
In particular \cite{TechreportIsomap} 
show that under suitable regularity
assumptions, the Euclidean path distance in the neighborhood graph is (with high probability with respect to
the data sampling) a good approximation of
the geodesic distance on $\mathcal{M}$; 
and \cite{Belkin2003} 
highlight the central role of the
graph Laplacian operator, and its spectral decomposition, as a fundamental data analysis tool. Over the years a rich
literature has developed exploring the mathematical properties of these objects. A central point of interest is to quantify to which extent the properties of the discrete random graph reflect those of the underlying submanifold, and further if convergence occurs in a suitable sense as the number of sampled points $n$ grows to infinity.

Our primary contribution in this work is to establish that two fundamental geometric properties, namely
a volume doubling condition (VD) and a local Poincaré inequality (LPI) hold for the random geometric graph
(with high probability, and uniformly over all shortest path distance balls in a certain radius range)
under suitable regularity conditions of the underlying submanifold $\mathcal{M}$ and the sampling distribution $\mathbb{P}$.
Informally speaking, we assume that $\mathcal{M}$ is compact without boundary and with bounded curvature in a
suitable sense; for (VD) we assume that $(\mathcal{M},\mathbb{P})$, as a metric measure space, itself satisfies (VD); and for (LPI), we consider the stronger assumption that $(\mathcal{M},\mathbb{P})$ is Ahlfors-regular.

The main motivation for establishing these two properties is that they imply \cite{Delmotte1999,Barlow2016} a precise sub-Gaussian estimate for the heat kernel on the graph, which is generated by the (random walk) graph Laplacian. In turn,
such estimates allow to establish the spatial localization properties of a graph wavelet construction based on
the spectral decomposition of the graph Laplacian, proposed in  \cite{Goebel2018} 
and following a general construction due to \cite{Coulhon2012}, where the sub-Gaussian estimate plays a central role.
We will return to these issues for a more detailed discussion in Section~\ref{sec:heatkernel}.

The outline of the rest of the paper is as follows. 
In Section~\ref{sec:results} we present the setting, the notations and the main results of this paper.
Sections~\ref{sec:prelim}, \ref{VD}, and~\ref{Poin} are devoted to prove the main results.  Section~\ref{sec:prelim} deals with the distance approximation of $\distb{G,SP}$ and $\distb{\manifold{M}}$ which is fundamental for our results 
, and presents further preliminary results.
In Section~\ref{VD} we establish \cref{theo::vdoka} concerning the (VD) property.
In Section~\ref{Poin} we present the proof of~\cref{theo::LPIspPi2} concerning the (LPI) property.

\section{Main results}\label{sec:results}

\subsection{Setting, basic notations and goals}\label{sec:setting}
We consider a specific class of geometric graphs. To be more precise we consider unweighted random $\varepsilon$-graphs whose vertex set is assumed to be a finite random sample from a submanifold of the Euclidean space $\RZ^K$ (considered as metric measure space $(\manifold{M}, \distb{\manifold{M}}, \mu)$) satisfying some properties introduced later on.
We will now present the setting and the notation used throughout this paper. \\



A finite geometric graph $G=(\mathcal{V},\mathcal{E})$ consists of a finite vertex set $\mathcal{V}=(x_1,\ldots, x_n)\subset\RZ^K$ ($x_i$ are assumed to be distinct) and an edge set $\mathcal{E}\subset \mathcal{V}\times \mathcal{V}$. We will use the notation $\mathcal{V}(G)$ and $\mathcal{E}(G)$ to denote the vertex and edge set for a specific graph $G$.
The graph can be described by its adjacency matrix $\mathbf{A}=(a_{ij})$ with $a_{ij}=1$ if there is an edge $e$ from $x_i$ to $x_j$ (denoted by $x_i\sim x_j$) and $a_{ij}=0$ otherwise.  
We will only consider undirected graphs without self-loops, that is
 $a_{ij}=a_{ji}$ for $i\neq j$  and $a_{ii}=0$ ($\mathbf{A}$ is symmetric).
The degree of a vertex $x_i$ is $\grad(x_i)=\sum_{x_j\in \mathcal{V}}a_{ij}$.
We denote the minimal degree in the graph  $\grad_{min}:=\min_{x\in \mathcal{V}} \grad(x)$ and the maximal degree $\grad_{max}:=\max_{x\in \mathcal{V}} \grad(x)$. 

\nomenclature[Cz]{$G=(\mathcal{V},\mathcal{E})$}{graph, vertex set, edge set} 
\nomenclature[Cz]{$e$}{edge}
\nomenclature[Cz]{$\mathbf{A}=(a_{ij})$}{adjacency matrix}  
\nomenclature[Cz]{$\grad$}{degree of a vertex}
\nomenclature[Cz]{$\grad_{min}, \grad_{max}$}{minimaler, maximaler grad im Graphen}  

We focus on neighborhood graphs, especially on $\varepsilon$-graphs, that is, for the construction of the edges the $\varepsilon$-rule is used. Two vertices (points)  are connected when their Euclidean distance (denoted $\distb{E}$) is smaller than $\varepsilon$: 
\begin{align} (x_i,x_j)\in \mathcal{E(G)} \Leftrightarrow a_{ij}=1~ \Leftrightarrow~ \dist{E}{x_i,x_j}= \norm{x_i-x_j}\leq \varepsilon.\label{constr}\end{align} 
By construction $\varepsilon$-graphs are undirected graphs. 
\nomenclature[Cz]{$\varepsilon$}{epsilon for $\varepsilon$-graph} 
\nomenclature[Cz]{$\distb{E}$}{Euclidean distance} 

A natural distance on the graph is the so-called shortest-path-distance $\distb{G,SP}$ ($\distb{SP}$ for short). A path $p$ from $x$ to $y$ in $G$ is a finite sequence $p=(v_0,\ldots,v_l)$ of vertices $v_i\in \mathcal{V}(G)$ with  $x=v_0, y=v_l$ satisfying $v_{i-1}\sim v_i$ for $i=1,\ldots,l$.
We denote $\NE{p}=l$ the number of edges of a path $p$ in a graph. Let $\mathcal{P}_{x,y}$ be the set of all paths in $G$ connecting $x$ to $y$.
We define for $x,y\in \mathcal{V}(G)$ the shortest-path-distance
\begin{align*}\dist{G,SP}{x,y}:=\min_{p\in \mathcal{P}_{x,y}} \NE{p}.\end{align*}
We denote for all $x\in \mathcal{V}$ by $\oB{x}{r}{G,SP}$ the open ball in $\mathcal{V}$ of radius $r$ centered in $x$ for the shortest-path distance and by $\cB{x}{r}{G,SP}$ the corresponding closed ball.
Note that the maximal shortest-path distance of two points in a connected graph with $n$ vertices is at most $n$, so that $\cB{x}{n}{G,SP}=\mathcal{V}$. 
\nomenclature[Cz]{$p$}{path in a graph} 
\nomenclature[Cz]{$\mathcal{P}_{x,y}$}{set of all paths connecting $x$ and $y$} 
\nomenclature[Cz]{$\NE{p}$}{number of edges of a path} 
\nomenclature[Cz]{$\distb{G,SP}$}{shortest path distance} 

Let $\eta$ be a discrete probability measure on the vertex set $\mathcal{V}$ of a graph $G$ determined by the point measure $\eta(x)$ for all $x\in \mathcal{V}$. 
Especially, we denote $\eta_1$ the empirical graph measure 
(based on number of vertices) \begin{align}\eta_{1}(W):=\frac{1}{n} \sum_{i=1}^{n}\Indik_{x_i\in W} \quad \text{~~for any~~} W\subset \mathcal{V},
\end{align}
where $\Indik_{x\in W}=1$ if $x\in W$ and $0$ otherwise,
and $\eta_2$  the (normalized) degree volume graph measure (
based on number of edges)
\begin{align}\label{eta2}
\eta_{2}(W):= \frac{\vol{W}}{\vol{\mathcal{V}}}
 \quad\text{~~for any~~}  W\subset \mathcal{V},
\end{align}
with $\vol{W}:=\sum_{i=1}^{n} \Indik_{x_i \in W} \grad(x_i)$.

With some abuse of notation we will also denote $\eta$ the measure on $(\manifold{M}, \mathcal{A})$ for any superset $\mathcal{M}\supseteq \mathcal{V}$ and $\sigma$-algebra
containing all singletons of $\mathcal{V}$.
\nomenclature[Cz]{$\eta$}{ measure on graph vertices} 
\nomenclature[Cz]{$\eta_1$}{empirical graph measure}
\nomenclature[Cz]{$\eta_2$}{degree volume graph measure}
  
We now recall the definitions of the volume doubling condition (VD) and the local Poincar\'{e} inequality (LPI) for graphs, the properties we want to prove in our specific setting.
Since $(G, \distb{SP}, \eta)$ is a metric measure space the following definition of volume doubling applies.
\begin{defi}[volume doubling]\label{def::vd}
	Let $(M,\distb{M}, \mu)$ be a metric  measure space with distance $\distb{M}(\cdot,\cdot)$
	and measure $\mu$. 
	Then $(M,\distb{M}, \mu)$ is said to satisfy the volume doubling condition (VD($v$))
	 if there exists a constant $v>0$ ($2^v\geq 1$) such that for all $x\in M$ and for all $r>0$:
	\begin{align}\label{doubl}
	0<\mu( \oB{x}{2r}{\distb{M}})\leq 2^v ~ \mu( \oB{x}{r}{\distb{M}})< \infty  
	\end{align}
	(where $\oB{x}{r}{\distb{M}}$ denotes the open ball in $M$ of radius $r$ centered at $x$  w.r.t.\ $\distb{M}$).
	
	The space $(M,\distb{M}, \mu)$ satisfies the restricted volume doubling condition\linebreak (rVD[$v ,r_{min}, r_{max}$]) if the volume doubling condition is satisfied with parameter $v$ for all balls $\oB{X}{r}{\distb{M}}$ with $x\in M$ and $r_{min}<r<r_{max}$. 
\end{defi}

\nomenclature[Cz]{$u$}{doubling constant (graph)} 

\begin{defi}[(weak) Local Poincar\'e Inequality for graphs]\label{defi-lpi-Barlow}
		A graph $G=(\mathcal{V},\mathcal{E})$ satisfies a local Poincar\'{e} inequality LPI($\lambda, C_{\lambda}, r_{max}$) with $\lambda\geq 1$, $C_{\lambda}\in\RZ, r_{max}>0$ if
		\begin{align*}
	\sum_{x\in \cB{x_0}{r}{}} \grad(x) \abs{f(x)-\overline{f}_{\cB{x_0}{r}{}}}^2 \leq C_{\lambda} r^2 \sum_{\substack{x,y \in \cB{x_0}{\lambda r}{},\\ x\sim y}}  (f(y)-f(x))^2
	\end{align*}
	\[(\text{where~} \overline{f}_{\overline{B}}:=\left(\vol(\overline{B})\right)^{-1}\sum_{x\in \overline{B}} \grad(x) f(x) \text{~and~} \vol(\overline{B})=\sum_{y\in \overline{B}} \grad(y))\]
 holds for all $f\in \RZ^{\abs{\mathcal{V}}}$ and for all closed balls $\cB{x_0}{r}{}=\cB{x_0}{r}{G,SP}$ of radius $r\leq r_{max}$ centered in $x_0 \in \mathcal{V}$ of $G$ w.r.t.\ $\distb{G,SP}$. 
		The LPI is called strong if $\lambda=1$.
\end{defi}
\nomenclature[Cz]{$C_{\lambda}$}{LPI constant (def)} 

Important to note is that in this paper the vertex set $\mathcal{V}$ consists of $n$ points which are drawn independent and identically distributed from  a  metric measure space $(\manifold{M}, \distb{\manifold{M}}, \mu)$ satisfying $\manifold{M}\subset\RZ^K$ with respect to the probability measure  $\mu$. 
Many properties of the graph will therefore depend on the properties of $(\manifold{M}, \distb{\manifold{M}}, \mu)$.
Most statements in this paper are probabilistic statements,  relative to the random sample of vertices.
Since the vertex set is random, the graph itself is random and  the graph measures are random variables.
Our main goal is to establish that, when the underlying space $\manifold{M}$ has suitable geometric regularity properties, then the random neighborhood graph satisfies the above properties (VD) and (LPI) with high probability growing to 1 as $n \rightarrow \infty$, 
with fundamental scaling constants $(v,C_\lambda)$ not depending on $n$. 
\nomenclature[Cz]{$(\manifold{M}, \distb{\manifold{M}}, \mu)$}{manifold, geodesic distance, measure} 
\nomenclature[Cz]{$k$}{dimension of data}
\nomenclature[Cz]{$\pi$}{pi}

\subsection{Asymptotics}
We consider $\varepsilon$-graphs constructed from a sample of size $n$.
We are interested in how the constants in our results depend on the parameters $n$ and $\varepsilon$ and what happens in the limit for $n$ going to infinity. The parameter $\varepsilon=\varepsilon(n)$ will be a 
 decreasing sequence going to zero, but not too fast. To be more precise we consider as standard asymptotics the case that \cref{ass::Ahlfors} (see next section)
is satisfied and that
\begin{align}\label{asym}
\varepsilon(n)\rightarrow 0 \text{~and~}
\frac{n\varepsilon(n)^k}{\ln(n)}\rightarrow \infty \text{~for~}n\rightarrow \infty
\end{align}
holds where $k$ is the intrinsic dimension of the underlying space and parameter of the Ahlfors regularity. This means that $\varepsilon(n)^k\in \Omega\left(\frac{\ln(n)}{n}\right)$.
This assumption will ensure that some probabilities occurring later on will converge to 0 and that some conditions on $n$ and $\varepsilon$ we encounter to guarantee some properties of the graph will be satisfied for $n$ large enough. 
In fact, \cref{asym} is a sufficient condition (in $n,\varepsilon$) for 
$n^\gamma \exp\left(-c n \varepsilon^k \right)$ and $\varepsilon^{-\gamma} \exp\left(-c n \varepsilon^k \right)$ tending to zero for $n\rightarrow \infty$, for any fixed positive constants $c,\gamma$. 

Under the standard asymptotics we have $\grad{x}\approx n\varepsilon^k \geq \ln n$ for $n$ large enough.
A slightly stronger assumption would be to assume that 
$n\varepsilon^k \geq n^{z}$ for some $z\in (0,1)$ for $n$ large enough (since $\bigO(\ln(n))\subseteq \bigO(n^{z})$).
Note that under the Ahlfors assumption (\cref{ass::Ahlfors}) $z$ is a measure of the connectivity of the graph: $z=0$ implies that no vertex is isolated (each vertex has at least one neighbor) and $z=1$ describes the fully connected graph. We are interested in a connected graph which reveals the local geometry of the underlying manifold.

\subsection{Main results}\label{sec::mainRes}

We formulate the following assumptions in order to state our results.

\nomenclature[Cz]{$c_l, c_u$}{Ahlfors constants} 
\nomenclature[Cz]{$r$}{ball radius} 

\nomenclature[Cz]{$k$}{dimension of submanifold} 
\begin{ass}

	\begin{assenum}[align=left,labelwidth=60pt,labelsep=0pt, leftmargin=*]
	
\item \label{ass::manifold}  
$\manifold{M}$ 
is a $k$-dimensional smooth compact submanifold of $\RZ^K$ without boundary,  with induced geodesic distance $\distb{\manifold{M}}$ and a Borel probability measure $\mu$  
with support in $\manifold{M}$. We define (as in  \cite{TechreportIsomap}) the  minimum radius of curvature $\crr$ and  the minimum branch separation $\cs$ of $\manifold{M}$ as
\begin{align*}\crr&:= \left(  \max_{\gamma,t}\norm{\ddot{\gamma}(t)} \right)^{-1} ~~\text{~ (with $\gamma$ being unit-speed geodesics) ~}\\
\cs:&=\max \{s : s>0 \text{~and~}  \norm{x-y} < s \Rightarrow \dist{\manifold{M}}{x,y}\leq \pi \crr  \text{~for~}  x,y\in \manifold{M}\}.
\end{align*}
Furthermore, let $\secK(\manifold{M})$ denote the sectional curvature of $\manifold{M}$ and $i(\manifold{M})$ the injectivity radius of $\manifold{M}$.

		
			\item \label{ass::graph} 
		$G=(\mathcal{V},\mathcal{E})$ is a (connected) undirected unweighted $\varepsilon$-graph 
		without self-loops with $\mathcal{V}$ consisting of $n$ vertices. 
		
		\item \label{ass::isomap} 
		For fixed parameters $\lambda_1, \lambda_2, \pr{1} \in (0,1)$ 
		 \begin{enumerate}[label=\alph*) , labelsep=3pt]
		 	\item the $\varepsilon$-graph is built with parameter 
		 	$\varepsilon>0$ 
		 	satisfying
		 	\begin{align}\label{A3eps}
		 	\varepsilon<\min\left(\cs,  (2/\pi)\crr \sqrt{24 \lambda_1}\right),
		 	\end{align}
		 	\item 
		 	the  sample size $n=n(\pr{1},\lambda_2,\varepsilon, \mu)$ satisfies 	
		 	\begin{align} \label{A3n}
		 	n \geq \frac{-\ln(\pr{1} u)}{u} \text{~~with~~~} u=
		 	\inf_{z\in\manifold{M}} \mu(B(z, \varepsilon \lambda_2 / 16)).
		 	\end{align} 
		 \end{enumerate}
	\item \label{ass::vd}  
		$(\manifold{M}, \distb{\manifold{M}}, \mu)$ satisfies the  volume doubling condition with constant $v$:\\
		there exists a constant $v>0$ ($2^v\geq 1$) such that for all $x\in \manifold{M}$ and for all $r>0$
		\begin{align}\label{doubl2}
		0<\mu( \oB{x}{2r}{\manifold{M}})\leq 2^v ~ \mu( \oB{x}{r}{\manifold{M}})< \infty  
		\end{align}
		where $	\oB{x}{r} {\manifold{M}}$ denotes the open ball of radius $r$ and centered in $x$ w.r.t.\ $\distb{\manifold{M}}$.
		\item \label{ass::Ahlfors} 
	$\mu$  is $k$-Ahlfors-regular:\\ 
		there exist constants $0<c_l\leq c_u<\infty$ such that for all $x\in \manifold{M}$ and all $r\in \left(0,\diam{\manifold{M}}\right]$
	\begin{align}
	c_l r^k \leq \mu(\oB{x}{r}{\manifold{M}}) \leq c_u r^k
	\end{align}
	holds.
	\item \label{ass::secCurv}
	The sectional curvature of $\manifold{M}$ is bounded:
	\begin{align}\secK(\manifold{M})\leq \Lambda < \infty.\end{align}
	We define \begin{align}r_{\bullet}:=\min \left(\frac{i(\manifold{M})}{2},\frac{\pi}{2\sqrt{\Lambda}}\right).\end{align}
	\end{assenum}
\end{ass}
\nomenclature[Cz]{$\lambda_1, \lambda_2$}{accuracy of distance approximation}
\nomenclature[Cz]{$n$}{sample size, number of points in $\mathcal{V}$}
\nomenclature[Cz]{$\pr{1}$}{probability related to sampling condition}
\nomenclature[Cz]{$d$}{dimension}
\nomenclature[Cz]{$L_{min}^*$}{lower bound bilipschitz}   
\nomenclature[Cz]{$\Lambda$}{upper bound of sectional curvature }
\Cref{ass::manifold,ass::graph} describe our setting. \Cref{ass::isomap} enables us to use a distance approximation proposed in \cite{TechreportIsomap} which is of importance for both the (VD) result and the local Poincar\'e inequality. For the (VD) result we need \cref{ass::vd}. The stronger \cref{ass::Ahlfors} and additionally \cref{ass::secCurv} are necessary for (LPI). Furthermore the dimension $k$ of $\manifold{M}$ and the parameter $k$ of the Ahlfors regularity will coincide.
 \begin{remark}\label{rem:A3}
   In our asymptotic regime \cref{ass::isomap} is satisfied for $n$ large enough with probability tending to 1. Note that in the asymptotic regime $\varepsilon(n)\rightarrow 0$ for $n$ going to infinity and therefore \cref{A3eps} is satisfied for $n$ large enough. Furthermore,  if we set $\pr{1}:=n^{-z}$ for $z>0$ we observe that under Ahlfors regularity the condition \cref{A3n} is satisfied if  \[n \geq \frac{ \ln(C_1^{-1} n^z\varepsilon^{-k})}{C_1\varepsilon^k} 
   \geq\frac{-\ln(\pr{1}u)}{u}\]
   with constant $C_1=c_l (\lambda_2 /16)^k$. Since the standard asymptotic regime implies $\varepsilon^{-k}\leq n$ for $n$ big enough, we obtain $ C_1^{-1} \ln(C_1^{-1} n^z\varepsilon^{-k}) \leq c \ln(n) \leq n\varepsilon^{k}$ for $n$ big enough for some constant $c$. This means that our standard asymptotic regime implies \cref{A3n}.
    \end{remark}
We can now state our first result.

\begin{theorem}[restricted Volume Doubling]
	\label{theo::vdoka}
		Let $G=(\mathcal{V}, \mathcal{E})$ be an $\varepsilon$-graph defined from an i.i.d. sample of size $n$ from  the probability  measure $\mu$ on the submanifold $\manifold{M}$ of $\RZ^K$ such that \cref{ass::manifold,ass::graph,ass::isomap,ass::vd} are satisfied with parameters  $\lambda_1\in (0,1), \lambda_2\in (0,1), \pr{1}\in (0,1), \varepsilon>0, n\geq 4, n\in \NZ, v>0$.
Let $\pr{2}\in (0,0.5]$.
Let $\eta_1$ 
be the empirical graph measure.
  
Then, 
with probability at least $1-\pr{1}-\pr{2}$, there exists a constant $u=u(v)>0$ such that for all open balls $B=B_{G,SP}(X_i,r)$ with $X_i\in \mathcal{V}$ and $r>1$ 
satisfying
$\eta_1\left(\oB{X_i}{r}{G,SP}\right)\geq 8 n^{-1}\ln\left(\frac{3n^2}{\pr{2}}\right)$ 
the inequality
\begin{align}
\eta_1(\oB{X_i}{2r}{G,SP}) \leq 2^u \eta_1 (\oB{X_i}{r}{G,SP})
\end{align}  
holds. The constant is $u:=\log_2(6)+\ceil*{4+ \log_2 (1+\lambda_2)(1-\lambda_1)^{-1}}v$. 
\end{theorem}

The condition $\eta_1\left(\oB{X_i}{r}{G,SP}\right)\geq 8 n^{-1}\ln\left(\frac{3n^2}{\pr{2}}\right)$ 
means that the doubling condition only applies to balls containing at least of the order of $\ln(n)$ points. In particular, under the standard asymptotics, for $n$ big enough, it will be satisfied for any $r>1$.

\begin{cor}\label{cor::vd-emp}
	Under the "standard asymptotics", the graph $G$ satisfies
	(rVD[$u, r_{min}=1, r_{max}=\infty$]) (with $u$ as above)  w.r.t.\ the empirical graph measure with probability going to 1 (and even probability $p$ such that $1-p=o(n^{-2})$).
\end{cor}

Note that we stated the main result for open balls. 
We can formulate the volume doubling Theorem also in terms of closed balls. (The theorem for open balls implies the theorem for closed balls, but not vice versa.) We get the result for closed balls as corollary.

\begin{cor}
		Let $G=(\mathcal{V}, \mathcal{E})$ be an $\varepsilon$-graph defined from an i.i.d. sample of size $n$ from  the probability  measure $\mu$ on the submanifold $\manifold{M}$ of $\RZ^K$ such that \cref{ass::manifold,ass::graph,ass::isomap,ass::vd} are satisfied with parameters  $\pr{1}\in (0,1), \varepsilon>0, n\geq 4, n\in \NZ, v>0$.
	Let $\pr{2} \in (0,0.5]$. 
	Let $\eta_1$ 
	be the empirical graph measure.
	
	Then, with probability at least $1-\pr{1}-\pr{2}$,  there exists a constant $u=u(v)>0$ such that for all closed balls $B=\cB{X_i}{r}{G,SP}$ with $X_i\in \mathcal{V}$ and $r\geq 1$ 
	satisfying
	$ \eta_1\left(\oB{X_i}{r}{G,SP}\right)\geq 8 n^{-1} \ln\left(\frac{3n^2}{\pr{2}}\right)$
	the inequality
	
	\begin{align}
	\eta_1(\cB{X_i}{2r}{G,SP}) \leq 2^u \eta_1 (\cB{X_i}{r}{G,SP})
	\end{align}  
	holds. 	
	The constant is $u:=\log_2(6)+\ceil*{4+ \log_2 (3)}v$. %
\end{cor}
This corollary is straightforward if we note that for any $r\geq 1$ exists $\delta(r)>0$ small enough such that $\cB{x}{2r}{SP}=\oB{x}{2r+2\delta}{SP}$ holds.

\nomenclature[Cz]{$\pr{2}$}{probability, lemma diff measure} 
\nomenclature[Cz]{$w$}{constant} 

Under the additional assumption that the graph is ``quasi-regular'' in the sense that the degrees are all of the same order up to a fixed constant, the volume doubling property holds also for the degree volume graph measure. This is stated in the following theorem.
	\begin{theorem}\label{vd-deg}
			Let $G=(\mathcal{V}, \mathcal{E})$ be an $\varepsilon$-graph defined from an i.i.d. sample of size $n$ from  the probability  measure $\mu$ on the submanifold $\manifold{M}$ of $\RZ^K$ such that \cref{ass::manifold,ass::graph,ass::isomap,ass::vd} are satisfied with parameters  $\pr{1}\in (0,1), \varepsilon>0, n\geq 4, n\in \NZ, v>0$.
Let $\pr{2}\in(0,0.5]$.
Let 
$\eta_2$ the degree volume graph measure.
Assume that,  with probability at least $1-\pr{3}$,  
we have $\frac{\max_{x\in \mathcal{V}}\grad_x}{ \min_{x\in \mathcal{V}} \grad_x }\leq c_{\bullet}$.
\nomenclature[Cz]{$c_{min}, c_{max}$}{upper and lower bound of vertex degrees }

Then 
with probability at least $1-\pr{1}-\pr{2}-\pr{3}$ there exists a constant $\tilde{u}=\tilde{u}(v)>0$ such that for all open balls $B=B_{G,SP}(X_i,r)$ with $X_i\in \mathcal{V}$ and  $r>1$
satisfying

$ \eta_1\left(\oB{X_i}{r}{G,SP}\right)\geq 8 n^{-1}\ln\left(\frac{3n^2}{\pr{2}}\right)$ 
the inequality
\begin{align}
\eta_2(\oB{X_i}{2r}{G,SP}) 
\leq 2^{\tilde{u}}\;\eta_2 (\oB{X_i}{r}{G,SP})
\end{align}  
 holds. The constant is $\tilde{u}=\log_2(6)+\ceil*{4+ \log_2(3)}v+2\log_2(c_{\bullet})$. 
\end{theorem}

\begin{cor}\label{cor::vd-deg}
	Under the "standard asymptotics", the graph $G$ satisfies
	(rVD[$\tilde{u},r_{min}=1, r_{max}=\infty$]) (with $\tilde{u}$ from above) w.r.t.\ the degree volume graph measure with probability going to 1 (and even probability $p$ such that $1-p=o(n^{-2})$).
\end{cor}

We prove \cref{theo::vdoka,cor::vd-emp,vd-deg,cor::vd-deg} in \cref{VD}. Now we present our results regarding the local Poincar\'{e} inequality for random $\varepsilon$-graphs.

\begin{theorem}[LPI, degree measure]\label{theo::LPIspPi2}%
	Let $G=(\mathcal{V}, \mathcal{E})$ be an $\varepsilon$-graph defined from an i.i.d. sample of size $n$ from  the probability  measure $\mu$ on the submanifold $\manifold{M}$ of $\RZ^K$ such that \cref{ass::manifold,ass::graph,ass::isomap,ass::Ahlfors,ass::secCurv} 
	are satisfied with parameters  $\lambda_1, \lambda_2, \pr{1}, \varepsilon>0, n\geq 2, 0<c_l\leq c_u<\infty, k>0, r_{\bullet}$.
	Consider the volume degree graph measure $\eta_2$.
	Let $L_{min}^*$ and $L_{max}^*$ with $0<L_{min}^*\leq L_{max}^*<\infty$ denote the Lipschitz constants independent of $n,\varepsilon,r$ induced by \cref{ass::secCurv}.
	Let  $\delta\in (0,1)$.
	Assume $r_{\bullet}(1-\lambda_1)\varepsilon^{-1}\geq 1$, 	$n\geq \frac{1}{(1-\delta)c_l} \left(\frac{4 \sqrt{k+3} L_{max}^*}{L_{min}^* \varepsilon}\right)^k +1$ 
	and $\frac{\sqrt{k+3}}{L_{min}\varepsilon}\geq 1$,  
	and define
	\[\pr{4}:= 2	\left(\frac{2 \sqrt{k+3}n (1-\lambda_1)}{L_{min}^*}\right)^k  \exp\left(-\frac{\delta^2 n c_l}{6} \frac{\varepsilon^k {L_{min}^*}^k}{4^k \sqrt{k+3}^k {L_{max}^*}^k}\right) +  2 \exp\left(-\frac{\delta^2 n c_l \varepsilon^k }{6(1-\lambda_1)^{k}}\right).\]
	
	Then  there exist  constants $\lambda>0$ and $\hat{C}>0$ such that, with probability at least 
	$1-n^2\pr{4}-\pr{1}$,
	for all balls $B=\overline{B}_{SP}(X_i,r)$ 
	with $r\in  \left[ 1,\min((1-\lambda_1)\varepsilon^{-1} r_{\bullet},n)  \right)$
	and $X_i\in \mathcal{V}$,   and for all functions $f: \mathcal{V}\rightarrow \RZ$ 
	the inequality
	\begin{align}
	\sum_{x\in \overline{B}_{SP}(X_i,r)} (f(x)-\overline{f}_{\overline{B}})^2 \grad(x) 
	&\leq \hat{C}  r^2 \sum_{\substack{x,y \in \overline{B}_{SP}(X_i,\lambda r)\\ x\sim y }}  (f(x)-f(y))^2 
	\end{align}
	holds.
	The constants are  $\lambda=4\frac{(1+\lambda_2)}{(1-\lambda_1) }+1$ and 
			\begin{align}
	\hat{C}:=& 
	\frac{1}{(1-\lambda_1)^{2+2k}}   \frac{(1+\delta)^2c_u^2}{(1-\delta)^2 c_l^2  } 	\left(1+w\right)^2 k^2 \left(2\frac{\sqrt{k+3}}{L_{min}^*}\right)^{k+2}\\
	=&const\left( \lambda_1,  c_l,c_u, w, k, \delta, L_{min}^*\right) 
	\end{align}
	with $w:=\frac{2(1+\delta)}{1-\delta}\frac{c_u}{c_l}\frac{{L_{max}^*}^k}{{L_{min}^*}^k}4^k \sqrt{k+3}^k.$ 	
\end{theorem}

\begin{cor}\label{cor::lpi-deg}
	Under the "standard asymptotics" the graph $G$ satisfies  LPI($\lambda, C_{\lambda}, r_+$) with $\lambda, C_{\lambda}=\hat{C}$ as in \cref{theo::LPIspPi2} and $r_+=\min(r_{\bullet}(1-\lambda_1)\varepsilon^{-1} ,n)$  with probability going to 1. 
\end{cor}

\begin{proof}[of \cref{cor::lpi-deg}]
	We set $\pr{1}=n^{-z}$ for some $z>0$.
	In our standard asymptotic regime, the probability $n^2\pr{4}$ converges to 0 and therefore $1-n^2\pr{4}-\pr{1}$ tends to 1 for $n$ going to infinity.
	Furthermore the conditions
	$r_{\bullet}(1-\lambda_1)\varepsilon^{-1}\geq 1$	
	and $\frac{\sqrt{k+3}}{L_{min}\varepsilon}\geq 1$ are satisfied, if $n$ is large enough since $\varepsilon(n)\rightarrow 0$.
	The inequality $n\geq \frac{1}{(1-\delta)c_l} \left(\frac{4 \sqrt{k+3} L_{max}^*}{L_{min}^* \varepsilon}\right)^k +1$ also holds in our asymptotic regime since $n\varepsilon^k$ will be larger than any constant for $n$ large enough.
\end{proof}

\Cref{Poin} is dedicated to the proof of \cref{theo::LPIspPi2}. We actually establish a slightly more general result in \cref{Poin} and can therefore state a local Poincar\'e inequality w.r.t\ the empirical graph measure in the following form. 

\begin{cor}\label{cor::lpi-emp}
		Under the "standard asymptotics" the graph $G$ satisfies a local Poincar\'e inequality w.r.t.\ the empirical measure
		of the form 
			\begin{align*}
		\sum_{x\in \overline{B}_{SP}(X_i,r)} (f(x)-\overline{f}_{\overline{B}})^2 
		&\leq \hat{C}  r^2 \sum_{\substack{x,y \in \overline{B}_{SP}(X_i,\lambda r),\\ x\sim y }} (f(x)-f(y))^2 
		\end{align*}
	 	(where $\overline{f}_{\overline{B}}=n_{\overline{B}}^{-1} \sum_{y\in \overline{B}} f(y)$) with probability going to 1.
		The constant is \begin{align*}\hat{C} 
		=&const\left( \lambda_1, c_l, L_{min}^*, k, w, \delta \right) \cdot \frac{1}{n\varepsilon^k}\\
		=&\frac{1}{\varepsilon^k n}\frac{1 }{(1-\lambda_1)^2}   \frac{(\eta^+)^2}{(1-\delta)c_l \eta^-} 	\left(1+w\right)^2 k^2 \left(2\frac{\sqrt{k+3}}{L_{min}^*}\right)^{k+2}
		\end{align*}
		with $w:=\frac{1+\delta}{1-\delta}\frac{c_u}{c_l}\frac{{L_{max}^*}^k}{{L_{min}^*}^k}4^k \sqrt{k+3}^k$.
\end{cor}

\subsection{Application: Heat kernel bound}
\label{sec:heatkernel}

We now return to a motivation for establishing (VD) and (LPI) mentioned in the introduction, namely obtaining
estimates for the heat kernel on the geometrical graph. We first summarize the results obtained in
\cite{Barlow2016} for a fixed graph. For coherence we will keep the notation used up until now.
Let $\mathbf{D}$ denote the diagonal matrix with $\mathbf{D}_{ii} := \grad{x_i}$,
and $\mathbf{L} := \mathbf{I}-\mathbf{D}^{-1}\mathbf{A}$ the so-called normalized random walk Laplacian.
Consider the continuous time random walk $X_t$ with generator $\mathbf{L}$, and
let $\mathbf{P}_t(x_i,x_j):=\mathbb{P}_{x_i}(X_t=x_j) = [\exp (-t\mathbf{L})]_{ij}$ denote probability point
function for $X_t$ starting from point $x_i$, then $\mathbf{Q}_t := \mathbf{P}_t \mathbf{D}^{-1}$ is called
heat kernel on the graph.

Theorem~1.2 of \cite{Barlow2016} implies the following.
Suppose the following assumptions are satisfied with respect to the shortest path distance and the degree volume graph measure $\eta_2$, for $1\leq r_{\min} \leq r_{\max}$:
\begin{enumerate}
\item the restricted Volume Doubling condition (with constant $v$), for all balls of radius in the range $[r_{\min} , r_{\max} ]$ ;
\item the weak local Poincaré inequality (with constants $c_\lambda,\lambda$), for all balls of radius in the range $[r_{\min} , r_{\max} ]$ ;
\item 
  the inequality $\vol(B(x,r)) \leq C_0 \vol({x}) r^v$, for all balls of radius in the range
  $[r_{\min} , r_{\max} ]$.
  
\end{enumerate}
Then the following estimate holds:
\begin{equation}
\label{eq:heatkernelbound}
\mathbf{Q}_t(x,y) \leq \frac{c_1}{\vol\left(\cB{x}{\sqrt{t}}{G,SP}\right)}
\exp\left( - c_2\frac{\dist{G,SP}{x,y}^2}{t}\right), 
\end{equation}
for all $t \in [r_{\min}^{1.1},r_{\max}]$ with $t \geq \dist{G,SP}{x,y}$.
Above $c_2$ is a {\em universal} constant and $c_1$ depends on $(v,c_\lambda,\lambda,C_0)$.

Given Corollaries~\ref{cor::vd-deg} and~\ref{cor::lpi-deg},
under the standard asymptotics we have that the first and the second
above conditions are satisfied with probability going to 1 as $n\rightarrow \infty$,
with constants $(v,c_\lambda,\lambda)$ {\em not} depending on $n$, $r_{\min}=1$
and $r_{\max}$ of the order $O(\varepsilon(n)^{-1})$ (which is the order of magnitude of the
graph diameter).

Finally, concerning the third assumption above, under the standard asymptotics, Theorem~\ref{theo::edegrees} point (iv) guarantees that with probability going to 1 as $n\rightarrow \infty$,
all degrees in the graph are uniformly upper and lower bounded up to constant factor by $n\varepsilon^k$; while inequality~\eqref{ballinclusion} together with Corollary~\eqref{lem::ubOka} ensure that the cardinality of ball $\cB{x}{r}{G,SP}$
is uniformly (in $x,r$) upper bounded up to constant factor by $nr^k\varepsilon^k$.
Therefore, the third assumption is satisfied with $C_0$ of order $n\varepsilon^k$.

We can thus apply the result of~\cite{Barlow2016}; however the interest of the obtained bound~\eqref{eq:heatkernelbound}
on the heat kernel rests on the dependence of the factor $c_1$ in the constant $C_0$
from the third assumption. The paper~\cite{Barlow2016} is not specific concerning this point,
and it is not obvious to track the dependence on the constants throughout all arguments there, but
we surmise that the dependence is at most a (small) power. In this case, the obtained
heat kernel bound is informative (because of the dominating exponential factor) as soon as $\distb{G,SP}(x,y) \geq C \sqrt{t} \log n$ for a sufficiently large constant $C$ (to be compared with the
graph diameter, of order $\varepsilon^{-1}$ which will typically be a power of $n$).

An important potential application of such heat kernel bounds is to establish spatial localization
properties of kernels based on spectral localization of the Laplacian.
Denoting ${\mathbf{L}}' := \mathbf{D}^{\frac{1}{2}} \mathbf{L} \mathbf{D}^{-\frac{1}{2}} $ the symmetrically normalized
graph Laplacian, $(\lambda_j,f_j)_{j\geq 1}$ its eigendecomposition, and an appropriate band-pass compactly supported function $\zeta:\RZ_+ \rightarrow \RZ_+$, it was proposed in \cite{Goebel2018} to construct a ``graph wavelet'' frame based on the spectrally localized kernels
\[
  \mathbf{K}_\ell := \sqrt{\zeta}_\ell(\mathbf{L}') = \sum_{j\geq 1} \sqrt{\zeta_\ell(\lambda_j)} f_jf_j^T, \qquad \zeta_\ell(x) = \zeta( 2^{-\ell}x), \ell \geq 1,
\]
(see also \cite{Hammond2011} for related work). An important desirable property of this construction
is the spatial localization of $\mathbf{K}_\ell(x,y)$ as a rapidly decaying function of $\varepsilon \distb{G,SP}(x,y) / 2^\ell$. Establishing such a theoretical property has been realized in a very general framework of Dirichlet
spaces by~\cite{Coulhon2012}, crucially using as an assumption a sub-Gaussian estimate for the
kernel of $[\exp (-t{\mathbf{L}'})]$ (when interpreted in the present setting). Note that under the standard asymptotics considered
in this paper, since $\mathbf{D}$ is upper and lower bounded by a multiple of identity up to constant
factor, the estimates for $\exp (-t{\mathbf{L}})$ are sufficient.

It is convenient and natural to use the rescaling $\tilde{\rho}:=\varepsilon{\distb{G,SP}}$ (which is
equivalent to $\distb{\manifold M}$, see next section), $\tilde{\mathbf{L}}:= \varepsilon^{-2}\mathbf{L}$, and $\tilde{t}=\varepsilon^2 t$. In this light, the above estimates
translate to estimates for $\exp (-\tilde{t}\tilde{{\mathbf{L}}})$, the ``large time'' condition in~\eqref{eq:heatkernelbound}
becomes $\tilde{t} \geq \varepsilon \tilde{\rho}(x,y)$. Although the theory developed in~\cite{Coulhon2012} requires in principle sub-Gaussian estimates for all $t\leq 1$, it seems plausible that the
obtained kernel localization estimates there still hold for ``larger'' scales $\ell \lesssim -\log \varepsilon$
when the sub-Gaussian estimates are restricted to $\tilde{t} \gtrsim \varepsilon$.

We finally comment on another plausible route to establishing the spatial localization properties of
$\mathbf{K}_\ell$: use the convergence (in a suitable sense), as $n\rightarrow\infty$, of $\tilde{\mathbf{L}}$, and of its spectral decomposition, to its continuous analogue the Laplace-Beltrami operator on $(\manifold{M},\distb{\manifold{M}}, \mu)$, for which the theory of~\cite{Coulhon2012} applies
directly. Up to our knowledge, the latest developments on this delicate subject by~\cite{Trillos2018}
establish convergence of the eigenfunctions in the $L^2(\mu)$ sense, which does not appear to
be strong enough to obtain the wished pointwise estimates. Additionally, we note that the
geometrical properties (VD) and (LPI) (and resulting heat kernel estimates) are more {\em robust} than convergence of the eingendecomposition, in the sense that they encode important regularity
properties of the geometrical graph rather than the Laplacian itself. In particular, it is observed in practice that the
localization properties discussed above hold qualitatively even if the eigenfunctions themselves are
clearly still far from convergence to their continuous counterparts.



\section{Preliminaries: On distance approximation and counting points}
\label{sec:prelim}
 \subsection{Distances}
 
  We will consider balls with respect to different metrics and as subsets of different spaces.
 
 Beside the shortest-path-distance introduced in \cref{sec:setting} we will use another graph-based distance. We define
 the Euclidean graph distance
  \[\dist{G,E}{x,y}:=\min_{p\in \mathcal{P}_{x,y}}\sum_{i=1}^{\NE{p}} \norm{v_{i-1}-v_i}
 \]
 as in \cite{TechreportIsomap}.

 We consider balls in a graph with respect to the shortest-path distance, the Euclidean graph distance and the metric $\distb{\manifold{M}}$. For all $x\in \mathcal{V}\subset \manifold{M}$ we denote
 \begin{align*}
 \oB{x}{r}{G,E}&=\{y\in \mathcal{V}: \dist{G,E}{x,y}< r\}\\
 \oB{x}{r}{G,SP}&=\{y\in \mathcal{V}: \dist{G,SP}{x,y}< r\}\\
 \oB{x}{r} {G,\manifold{M}}&= \{y\in \mathcal{V}: \dist{\manifold{M}}{x,y}< r\}=B_{\manifold{M}}(x,r)\cap \mathcal{V}. 
 \end{align*}
 
  Balls in the submanifold are defined using the geodesic distance $\distb{\manifold{M}}$ and the Euclidean distance respectively. We denote
 \begin{align*}
 \oB{x}{r} {\manifold{M}}&=\{y\in \manifold{M} : \dist{\manifold{M}}{x,y}< r\}~~  \forall x\in \manifold{M}, \\
 \oB{x}{r}{\manifold{M},E}&=\{y\in \manifold{M}; \dist{E}{x,y}:=\norm{x-y} <r \} ~~\forall x\in \manifold{M}.
 \end{align*}
 Similarly, we denote corresponding closed balls as $\cB{x}{r}{G,E}$, etc.

\subsection{Distance approximations}\label{sec::distapp}
The first step on our way to prove the main results is to establish a link between the shortest-path distance $\distb{G,SP}$ of the graph and the geodesic distance $\distb{\manifold{M}}$ of the submanifold. It was proved in \cite{TechreportIsomap} that under \cref{ass::manifold,ass::graph,ass::isomap}  $\distb{G,E}\approx \distb{\manifold{M}}$ holds with high probability. On the other hand we will prove that  $\distb{G,E}$ can be approximated by the shortest-path-distance in the graph ($\distb{G,SP}$).

We recall now the theorem from \cite[Main Theorem B]{TechreportIsomap} about the distance approximation $\distb{G,E}\approx \distb{\manifold{M}}$ and continue with the link between $\distb{G,E}$ and $\distb{G,SP}$.
\begin{theorem}[distance approximation 1]
	\label{theo::dist1}
		Let $G=(\mathcal{V}, \mathcal{E})$ be an $\varepsilon$-graph defined from an i.i.d. sample of size $n$ from  the probability  measure $\mu$ on the submanifold $\manifold{M}$ of $\RZ^K$ such that \cref{ass::manifold,ass::graph,ass::isomap} 
	are satisfied with parameters  $\lambda_1, \lambda_2, \pr{1}, \varepsilon>0, n\geq 2$.
	Then  it holds, with probability at least $1-\pr{1}$,
	for all $x,y\in \mathcal{V}(G)$
	
	\begin{align}(1-\lambda_1) \dist{\manifold{M}}{x,y} \leq \dist{G,E}{x,y}\leq (1+\lambda_2) \dist{\manifold{M}}{x,y}.\end{align}
\end{theorem}

\begin{remark}
		The relationship of the parameters $\pr{1}, n, \varepsilon, \lambda_1, \lambda_2$ is determined  by \cref{ass::isomap}: \linebreak $n\geq n_{min}(\pr{1}, \varepsilon, \lambda_2, \mu)$  (sampling condition), and $\varepsilon< \varepsilon_{max}(\cs,\crr, \lambda_1)$.
\end{remark}
For the proof see \cite[ proof of Main Theorem B]{TechreportIsomap}.

\begin{theorem}[distance approximation 2] 
	\label{thm::1} 
	\label{theo::dist2}
	Let $G$  be an $\varepsilon$-graph. For all $x,y \in \mathcal{V}(G)$, it holds
	\begin{align}\label{ineq::dist2}
	\frac{1}{4}\; \varepsilon \; \left(\dist{G,SP}{x,y}-1\right) \leq \dist{G,E}{x,y} \leq \varepsilon\; \dist{G,SP}{x,y}.\end{align}
\end{theorem}
\begin{proof}
	First we show  $\dist{G,E}{x,y}\leq \varepsilon \,  \dist{G,SP}{x,y}$. 
	For arbitrary $x,y \in \mathcal{V}$ let $\mathcal{P}_{x,y}$ the set of paths connecting $x$ to $y$ and $\EL{p}:=\sum_{i=1}^{\NE{p}}\norm{v_i-v_{i-1}} $ the Euclidean length of a path $p=(v_0,\ldots,v_{\NE{p}})$. 
	By definition we have 
	\[\dist{G,E}{x,y}=\min_{p\in\mathcal{P}_{x,y}}\EL{p}\leq \EL{q}\]
	for all $q\in\mathcal{P}_{x,y}$. For any path $q\in\mathcal{P}_{x,y}$ with $l$ edges we get
	\begin{align*}EL(q)=\sum_{i=0}^{l-1}\distb{E}(x_i,x_{i+1})\leq l \max_{i}\distb{E}(x_i,x_{i+1})\leq l \cdot  \varepsilon \end{align*}
	since for any edge in an $\varepsilon$-graph: $\distb{E}(x_i,x_{i+1})\leq \varepsilon$.
	Now we choose a path $q^*\in\mathcal{P}_{x,y}$ with minimal number of edges:$~~l^*=\min_{p\in\mathcal{P}_{x,y}}  \NE{p}=\dist{G,SP}{x,y}$.\\
	So we get $EL(q^*)\leq \varepsilon \cdot \dist{G,SP}{x,y}$.
	Summarized we have 
	$\dist{G,E}{x,y}\leq EL(q^*)\leq \varepsilon\cdot \dist{G,SP}{x,y}$.\\
	
	As a second step we show $\dist{G,E}{x,y}\geq 1/4 \cdot \varepsilon \; (\dist{G,SP}{x,y}-1) $.\\
	Let $x,y \in \mathcal{V}, x\neq y$ be given and assume that $x$ and $y$ are not neighbors (if $x=y$ or $x\sim y$, then the lower bound in \cref{ineq::dist2} is trivial). 
	We choose $p^*\in\mathcal{P}_{x,y}$ such that $p^*\in \argmin_{p\in\mathcal{P}_{x,y}}\EL{p}=S_{xy}$ and $\NE{p^*}=\min_{p\in S_{xy}}\NE{p}=l^*\geq 2$  (that is $p^*$ is a path with minimal number of edges in the set of paths with minimal  Euclidean graph distance). 
	Notice that there are no two adjacent edges of the path are smaller then $\varepsilon/2$; by contradiction if $v_{i-1},v_i,v_{i+1}$ were vertices with $\distb{E}(v_{i-1},v_{i+1} )\leq \distb{E}(v_{i-1},v_i)+\distb{E}(v_i,v_{i+1}) \leq \varepsilon/2 + \varepsilon/2=\varepsilon$ (hence $v_{i-1}\sim v_{i+1}$), then the path without $v_i$ would have smaller length. Therefore at least  $(l^*-1)/2 $ edges of the path $p^*$ have Euclidean length  $>\varepsilon/2$.
	
	Therefore $\sum_{i=0}^{l^*-1}\distb{E}(x_i,x_{i+1})\geq \frac{l^*-1}{2} \frac{\varepsilon}{2}$. \\Obviously $\NE{p^*}\geq \dist{G,SP}{x,y}$.
	
	
	 So we get
	\begin{align*}
	\dist{G,E}{x,y}=\EL{p^*}=\sum_{i=0}^{l^*-1}\distb{E}(x_i,x_{i+1})\geq \frac{1}{2\cdot 2}\varepsilon (l^*-1)\geq \frac{1}{4}\varepsilon \left(\dist{G,SP}{x,y}-1\right).
	\end{align*}
\end{proof}

Note that in particular $\overline{B}_{G,SP}(x,1) = \overline{B}_{G,E}(x,\varepsilon)$ holds which is obvious by the construction of the $\varepsilon$-graph (see \cref{constr}) and \cref{theo::dist2}.

As an immediate consequence  we can relate the manifold distance and the shortest path distance.

\begin{cor}\label{cor::dist}
		Let $G=(\mathcal{V}, \mathcal{E})$ be an $\varepsilon$-graph defined from an i.i.d. sample of size $n$ from  the probability  measure $\mu$ on the submanifold $\manifold{M}$ of $\RZ^K$ such that \cref{ass::manifold,ass::graph,ass::isomap} 
	are satisfied with parameters  $\lambda_1, \lambda_2, \pr{1}, \varepsilon>0, n\geq 2$.
Then  it holds, with probability at least $1-\pr{1}$,
	for all $x,y\in \mathcal{V}(G)$:
	\begin{align}
	\frac {(1-\lambda_1)}{\varepsilon} \dist{\manifold{M}}{x,y}&\leq \dist{G,SP}{x,y}  \leq \frac{4 (1+\lambda_2)}{\varepsilon} \dist{\manifold{M}}{x,y} +1.  \label{ineq::dist}
	\end{align}
	Consequently, with  probability at least $1-\pr{1}$, the inclusions
	\begin{align}\label{ballinclusion}
	\oB{x}{r}{G,SP} &\subseteq \oB{x}{(1-\lambda_1)^{-1}\varepsilon \; r}{G,\manifold{M}},\\
	\oB{x}{r} {G,\manifold{M}}& \subseteq \oB{x}{4\; \frac{1+\lambda_2}{\varepsilon} \; r+1}{G,SP}
	\end{align}
	hold for all $x\in \mathcal{V}$, $r>0$.
\end{cor}

\subsection{On counting points in sets}\label{counting}
We recall in this section some results stated in \cite{Ulrike2014} in our notation. \\
These results give bounds on the minimal and maximal number of sample points in a collection of subsets which includes the special case of minimal and maximal degree in a $\varepsilon$-graph.
They are based on a well-known concentration inequality for a binomial-distributed random variable
(see e.g. 
\cite[prop 2.4]{angluin77},\cite{chernoff52} or \cite{hoeffding63}) which we recall here.
\begin{theorem}[concentration inequalities for binomials]\label{concIneq}
	If $N\sim \BinV(n,p)$ (with $p\in[0,1]$), 
	then it holds
	\begin{align}
	&\forall \delta\in (0,1)~~~&\prob{N\leq (1-\delta) np}& 
	\leq  \exp\left(-\frac{1}{3}\delta^2 n p\right),\label{concIn}\\
	&\forall \delta\in (0,1]~~~&\prob{N\geq (1+\delta) np} &\leq \exp\left(-\frac{1}{3}\delta^2 n p\right).\label{concIn2}
	\end{align}
\end{theorem}
 Let $n_B$ denote the random number of points out of  $\{X_1, \ldots, X_n\} \subset \manifold{M}$ in an (open or closed) non-random ball $B=B_{\manifold{M},\dista}(x,r)$ for fixed $x$ and $r$ w.r.t.\ the metric $\dista$. Then $n_B$ is binomial distributed with parameters $n$ and $p=\mu(B)$. 
As a consequence of \cref{concIneq} the number $n_B$ of points in $B$ is bounded from below and above with high probability: to be precise the inequality 
\begin{align}\label{ineq::binomial}
(1-\delta) n \mu(B) \leq n_B \leq (1+\delta) n \mu(B)
\end{align} 
holds with probability at least $1-\pr{5}(B,n,\delta)$ (with $\pr{5}:=2 \exp\left(-\frac{\delta^2 n \mu(B)}{3} \right) $).

\begin{cor}\label{cor:counting}
	Let $\mathcal{V}$ be an i.i.d. sample of size $n$ of the  probability measure $\mu$ on the submanifold $\manifold{M}$ of $\RZ^K$.
	Let $B_1,\ldots,B_l$ be a collection of balls in $\manifold{M}$  and $N_i$ the  number of points in the ball $B_i$ and $w_i=\mu(B_i)$.
	Then for all $\delta\in(0,1]$ it holds
	\begin{align}
	\prob{N_{min} \leq (1-\delta) n w_-} \leq l\cdot \exp\left(-\frac{1}{3}\delta^2n w_-\right)\\
	\prob{N_{max} \geq (1+\delta) n w_+} \leq l\cdot \exp\left(-\frac{1}{3}\delta^2nw_- \right)
	\end{align}
	with $w_-:=\min_{i=1,\ldots,l} \mu(B_i)$ and  $w_+:=\max_{i=1,\ldots,l} \mu(B_i)$. 
\end{cor}

\begin{proof}
	This is a consequence of \cref{concIneq} and the union bound.
\begin{align*}
	\prob{N_{min}\leq (1-\delta)n w_-} \leq \sum_{i=1}^{l} \prob{N\leq (1-\delta)n w_i} \leq l \cdot \exp\left(-\frac{1}{3} \delta^2 n w_- \right)
	\end{align*}
	\begin{align*}
		\prob{N_{max}\geq (1+\delta)n w_+} \leq \sum_{i=1}^{l} \prob{N\leq (1+\delta)n w_i} \leq l \cdot \exp\left(-\frac{1}{3}\delta^2n w_- \right)
\end{align*}
\end{proof}
Now we allow for a random center point of the ball. To be more precise we consider random balls $B=B_{\manifold{M},\dista}(X_i, r)$. Then the random variable $n_B-1$ given $X_i$ is binomial distributed $\BinV(n-1, \mu(B))$. It follows that 
\begin{align}\label{ineq::binom2}
\prob{(1-\delta)  \mu(B) \leq \frac{n_B -1}{n-1} \leq (1+\delta) \mu(B)\given X_i}\geq 1-\pr{5}(B,n-1,\delta)
\end{align}
holds. Note that $\pr{5}:=2 \exp\left(-\frac{\delta^2 (n-1) \mu \left( B \right)}{3}\right) $ depends on $X_i$.

These results are the basis in order to get bounds for the vertex degrees in an unweighted $\varepsilon$-graph satisfying \cref{ass::graph}.
The degree $\grad(X_i)$ in an unweighted $\varepsilon$-graph is the number of neighbors of $X_i$. A random vertex point $X_j$ is a neighbor of $X_i$ if $X_j \in \overline{B}_{\manifold{M},E}(X_i,r)$.

This leads to the following bounds on the minimal and maximal degree of a vertex in the random $\varepsilon$-graph (cf. 
 \cite[prop. 29]{Ulrike2014}).
We use the following notation:
\begin{align*}
m_i&=\mu(\overline{B}_{\manifold{M},E}(X_i,\varepsilon)),\\ m_{min}&=\min_{i=1..n}\mu(\overline{B}_{\manifold{M},E}(X_i,\varepsilon)),\\ m_{max}&=\max_{i=1..n}\mu(\overline{B}_{\manifold{M},E}(X_i,\varepsilon)),\\ M&:=\min_{x\in \manifold{M}}\mu(\overline{B}_{\manifold{M},E}(x,\varepsilon)).
\end{align*}
Note that $m_i, m_{min}, m_{max}$ are random quantities and $M\leq m_{min}\leq m_i\leq m_{max}$.
\nomenclature[Cz]{$m_i$}{$\mu(\overline{B}_{\manifold{M},E}(X_i,\varepsilon))$, needed for bounds on degrees}
\nomenclature[Cz]{ $m_{min}$ }{$\min_{i=1..n}\mu(\overline{B}_{\manifold{M},E}(X_i,\varepsilon))$, needed for bounds on degrees}
\nomenclature[Cz]{$m_{max}$}{$\max_{i=1..n}\mu(\overline{B}_{\manifold{M},E}(X_i,\varepsilon))$, needed for bounds on degrees}
\nomenclature[Cz]{$M$}{$\min_{x\in \manifold{M}}\mu(\overline{B}_{\manifold{M},E}(x,\varepsilon))$, needed for bounds on degrees}
\begin{theorem}[degrees in $\varepsilon$-graph]
	\label{theo::edegrees}
	Let $G=(\mathcal{V}, \mathcal{E})$ satisfy \cref{ass::graph}
	Let $m_i, m_{min}, m_{max}
$ and $M$ be as defined above.

		Then  $\forall \delta\in (0,1]$ we have
	\begin{align*}
	\forall i=1,\ldots,n:~ \prob{\grad(X_i) \geq (1+\delta) (n-1) m_i\given X_i}
	& \leq   \exp\left(-\delta^2 (n-1) m_i/3\right)
	\end{align*}
		\begin{align}
	\prob{\grad_{max}\geq (1+\delta) (n-1) m_{max}}
	& \leq  n \exp\left(-\delta^2 (n-1) M /3\right)
	\end{align}
	and  $\forall \delta\in (0,1)$
	\begin{align*}
		\forall i=1,\ldots,n:~\prob{\grad(X_i) \leq (1-\delta) (n-1) m_i \given X_i} 
	& \leq   \exp\left(-\delta^2 (n-1) m_i/3\right)
	\end{align*}
		\begin{align} 
	\prob{\grad_{min}\leq (1-\delta) (n-1) m_{min}} & \leq  n \exp\left(-\delta^2 (n-1)  M/3\right).
	\end{align}
	If $n M /\ln(n) \rightarrow \infty$, these probabilities converge to $0$ as $n\rightarrow \infty$.
	\end{theorem}

\begin{proof}
	The first and third inequality are immediate consequences of the concentration inequalities \cref{concIn} and \cref{concIn2} since $\grad(X_i)$ given $X_i$ is binomial distributed with $n-1$ and $p=m_i$. 
 Now notice that the quantities $\grad_{max}$ and $\grad_{min}$ depend on $X_1,\ldots, X_n$. Therefore, to obtain the other two inequalities, the union bound, conditioning on $X_i$ and the previous results are applied:
	\begin{align*}
	\prob{\grad_{max}\geq (1+\delta) (n-1) m_{max}} 
	&\leq  \sum_i \prob{\grad(X_i)\geq (1+\delta) (n-1) m_{i}}\\
	& \leq \sum_i \ewx{X_i}{ \exp\left(-\delta^2 (n-1) m_i/3\right) } \\
	&\leq n \exp\left(-\delta^2   (n-1)  M /3\right),
	\end{align*}
	\begin{align*}
	\prob{\grad_{min} \leq (1-\delta) (n-1) m_{min}} 
	&\leq  \sum_i \prob{\grad(X_i)\leq (1-\delta) (n-1) m_{i}}\\
	& \leq \sum_i \ewx{X_i}{ \exp\left(-\delta^2 (n-1) m_i/3\right) }\\
	&\leq  n \exp\left(-\delta^2   (n-1)   M/3\right).
	\end{align*}
\end{proof}

The Ahlfors regularity of $\mu$ is especially of importance for the local Poincar\'{e} inequality. We will now present some consequences of this regularity assumption.

\begin{theorem}\label{theo::ahlf}
		Let $G=(\mathcal{V}, \mathcal{E})$ be an $\varepsilon$-graph defined from an i.i.d. sample of size $n$ from  the probability  measure $\mu$ on the submanifold $\manifold{M}$ of $\RZ^K$ such that \cref{ass::manifold,ass::graph,ass::isomap,ass::Ahlfors} 
	are satisfied with parameters  $\lambda_1, \lambda_2, \pr{1}, \varepsilon>0, n\geq 2, c_l,c_u,k$.
Let $\eta=\eta_1$ be the empirical graph  measure and denote $n_B$ the number of points in a set $B$, $\delta\in(0,1)$ 
and define $\pr{5}(B, n, \delta):=2 \exp\left(-\frac{\delta^2 n \mu(B)}{3}\right)$. Assumme that $\pr{5}(B,n, \delta)<1$.
\begin{enumerate}[label=\roman*)]
	\item Then $(\manifold{M}, \distb{\manifold{M}},\mu)$ satisfies the volume doubling condition (see \cref{def::vd}) 
	with constant $v = \log_2 (\frac{c_u}{c_l} )+k$. 
\item \label{part2}Then 
	for a fixed ball $B_{\manifold{M},E}(x,r)$ with $x\in \manifold{M}$ and $r\leq \varepsilon$
 the inequality 
	\[c_l  r^k \leq \mu(B_{\manifold{M},E}(x,r))\leq c_u (1-\lambda_1)^{-k} r^k\]
holds.
Moreover, for a fixed ball $B_{\manifold{M},E}(X_i,r)$ with $x\in \mathcal{V}$ and $r\leq \varepsilon$ the inequality
	\[c_l  r^k \leq \mu(B_{\manifold{M},E}(X_i,r))\leq c_u (1-\lambda_1)^{-k} r^k\]
	holds almost surely.
	\item \label{part3}
	
	Then 
	for any fixed ball $B=B_{\manifold{M}}(x,r)$ ($x\in \manifold{M},r>0$ ), it holds with probability at least $1-\pr{5}(B, n,\delta)$ that
	\[(1-\delta)  c_l r^k \leq \frac{n_B}{n} \leq (1+\delta)  c_u r^k. \]
	For any fixed ball $B=B_{\manifold{M}}(X_i,r)$ with random center point $X_i\in \mathcal{V}$ and radius $r>0$ it holds that
%
	\begin{align*}\prob{(1-\delta)  c_l r^k \leq \frac{n_B-1}{n-1} \leq (1+\delta)  c_u r^k }\geq 
	1-\pr{6}(r, c_l, n-1, \delta)
	\end{align*}
	where $\pr{6}(r, c_l, n-1, \delta):=2 \exp\left(-\frac{\delta^2 (n-1) c_l r^k}{3}\right)$.
		
	\item \label{part4} Then 
	for any fixed ball $B=B_{G,SP}(X_i,r)$ with random center point $X_i\in \mathcal{V}$ and radius $r\geq 2$ 
	it holds with probability at least 
		$1-\pr{6}\left(0.125(1+\lambda)^{-1} \varepsilon r, c_l,n-1, \delta\right)-\pr{1}$
	  that
	\[(1-\delta) c_l \left(\frac{1}{4 \cdot 2(1+\lambda_2 )} \varepsilon \right)^k r^k \leq \frac{n_{B_{G}}-1}{n-1} \leq (1+\delta) c_u \left(\frac{\varepsilon}{1-\lambda_1}\right)^k  r^k.\]
\end{enumerate}
\end{theorem}

\begin{proof}[of \cref{theo::ahlf}]
We assume the Ahlfors regularity of $\mu$:	\linebreak
$c_l r^k \leq \mu(B_{\manifold{M}}(x,r)) \leq c_u r^k$ for all $r\in \left(0,\diam{\manifold{M}}\right], x\in \manifold{M}$.
\begin{enumerate}[label=\roman*)]
	\item This is a well-known fact: 
	\[\mu(B(x,2r)) \stackrel{Ahlfors}{\leq} c_u (2r)^k \stackrel{Ahlfors}{\leq } \frac{c_u}{c_l} 2^k \mu(B(x,r)).\]
	\item	This follows immediately from Corollary 4 in \cite{TechreportIsomap} and the Ahlfors-assumption. The corollary states that under \cref{ass::isomap} 
	and $\distb{E}(x,y)\leq r\leq \varepsilon$: \linebreak $(1-\lambda_1) \dist{\manifold{M}}{x,y} \leq \distb{\manifold{M},E}(x,y)\leq \dist{\manifold{M}}{x,y}$ 
	which implies that \linebreak $ B_{\manifold{M}}(x,r) \subseteq B_{\manifold{M},E}(x,r) \subseteq B_{\manifold{M}}(x, (1-\lambda_1)^{-1} r) $ for $r\leq \varepsilon$. Considering a random center point $X_i$, we have
	\begin{align*}
	&\prob{c_l  r^k \leq \mu(B_{\manifold{M},E}(X_i,r))\leq c_u (1-\lambda_1)^{-k} r^k}\\
	&=\ewt[X_i]{\prob{c_l  r^k \leq \mu(B_{\manifold{M},E}(X_i,r))\leq c_u (1-\lambda_1)^{-k} r^k\given X_i}}=1
	\end{align*}
	by the first inequality.
		 
	\item 
	For a ball $B=B_{\manifold{M}}(x,r)$ with $x,r$ fixed (non-random) we have inequality \eqref{ineq::binomial} with probability at least $1-\pr{5}(B,n, \delta)$. Applying the Ahlfors assumption, 
	we get 
	with probability at least $1-\pr{5}(B,n,\delta)$
	\[(1-\delta) c_l r^k \leq \frac{n_B}{n} \leq (1+\delta) c_u r^k\]
	If the center point of the ball is one of the graph vertices then
	we use inequality \eqref{ineq::binom2} instead. Furthermore applying the Ahlfors assumption to bound $\pr{5}$ and integrating over $X_i$ leads to the unconditional probability:
	\begin{align*}
	&\prob{(1-\delta)  c_l r^k \leq \frac{n_B-1}{n-1} \leq (1+\delta)  c_u r^k }\\
	&\geq 	\ewt[X_i]{1-\pr{5}(B,n-1,\delta)} \geq 1-2 \exp\left(-\frac{\delta^2 (n-1) c_l r^k}{3}\right)
	\end{align*}

	\item 
	We consider $B=B_{G,SP}(X_i,r)$. Under assumption \ref{ass::isomap}  there are balls $B_1=B_{\manifold{M}}(X_i, \frac{\varepsilon}{4(1+\lambda_2)} (r-1) )$ and $B_2=B_{\manifold{M}}(X_i,\varepsilon (1-\lambda_1)^{-1} r)$  such that $X_j\in B_1\Rightarrow X_j \in B \Rightarrow X_j \in B_2$ with probability at least $1-\pr{1}$. For the numbers of vertices in $B,B_1$ and $B_2$ it follows that
	$n_{B_1}\leq n_B \leq n_{B_2}$. We now apply part \cref{part3}. Finally we use $r-1 \geq r/2$ for $r\geq 2$. Thus with probability 
		$1-\pr{6}\left(0.125(1+\lambda)^{-1} \varepsilon r, c_l,n-1, \delta\right)-\pr{1}$
	\begin{align*}
	(1-\delta) c_l \left(\frac{\varepsilon}{4 \cdot 2 (1+\lambda_2)}\right)^k r^k 
	& \leq  \frac{n_{B_1}-1}{n-1}
		 \leq  \frac{n_{B}-1}{n-1}\\
		& \leq  \frac{n_{B_2}-1}{n-1}
	 \leq (1+\delta) c_u \left(\frac{\varepsilon}{(1-\lambda_1)}\right)^k r^k.
	\end{align*}
\end{enumerate}
\end{proof}

\begin{remark}\label{rem:deg}
	Since we consider $\mu$ to be a Radon probability measure, it is inner and outer regular and the Ahlfors condition holds also for closed balls.
	Consequently \cref{theo::ahlf} holds also true for closed balls. 
Under Ahlfors regularity of $\mu$ we get therefore by \cref{theo::ahlf} \ref{part2} 
that \[m_i, m_{min}, m_{max}, M\in [c_l \varepsilon^k, c_u (1-\lambda_1)^{-k}\varepsilon^k].\]
and consequently \cref{theo::edegrees} gives
\[	\prob{\grad_{min}\leq (1-\delta) (n-1) c_l  \varepsilon^k}
 \leq   n \exp(-\delta^2  (n-1)  c_l \varepsilon^k/3) \]
 and
\[	\prob{\grad_{max}\geq (1+\delta) (n-1) c_u(1-\lambda_1)^{-k}\varepsilon^k}
\leq   n \exp(-\delta^2  (n-1)  c_l \varepsilon^k/3). \]
This implies that under the Ahlfors assumption the degrees are of order $n\varepsilon^k$.
\end{remark}

\section{Proofs of volume doubling results}\label{VD}


If we assume the Ahlfors condition  on $(M,d_M)$ (see \cref{ass::Ahlfors}) it is possible to establish a version of the volume doubling condition as a consequence of \cref{theo::ahlf}  (an additional argument to obtain the uniformity over all balls will be still necessary). In this section we aim at proving (rVD) without requiring Ahlfors regularity, only assuming (VD) of the underlying manifold.

The proof of \cref{theo::vdoka} is based on the
approximation of distances $\distb{\manifold{M}}, \distb{G,E}, d_{SP}$ (as introduced in  \cref{sec::distapp}), an uniform relative bound on $\abs{\mu(B)-\eta_1(B)}$ 
and the volume doubling property of $\mu$ on the manifold.
We now state our result concerning the uniform bound.
\begin{theorem} [uniform relative bound on $\abs{\eta_1(B)-\mu(B)}$ 
	]\label{theo::ubOka-a}
	Let $\pr{2} \in (0,0.5]$ and $\mathcal{V}$ be a random sample of size $n\geq 4$
	drawn independently and identically distributed from $\manifold{M}$ w.r.t.\ the measure $\mu$.
	
	Then, with probability at least 
	$1-\pr{2}$,    the inequality 
		\begin{align}\label{eq::oka-aa}
	\abs{\sqrt{\eta_1(\oB{X_i}{r}{\manifold{M}})}-\sqrt{\mu(\oB{X_i}{r}{\manifold{M}}) } } \leq 2\sqrt{\frac{-\ln(\frac{\pr{2}}{4 n^2})}{n}}
	\end{align}
	holds for all $X_i\in \mathcal{V}$ and for all $r>0$ ($r\in \RZ_+$).
	
\end{theorem}
We prove this theorem  at the end of this section. 
A simple consequence of \cref{theo::ubOka-a} is the following corollary.
\begin{cor}\label{lem::ubOka}
		Let $\pr{2} \in (0,0.5]$ and $\mathcal{V}$ be a random sample of size $n$
	drawn independently and identically distributed from $\manifold{M}$ w.r.t.\ the measure $\mu$.
	
		Then, with probability at least 
	$1-\pr{2}$, for all $X_i\in \mathcal{V}$ and for all $r>0$ ($r\in \RZ_+$)  \\ 
	the inequalities
	\[\eta_1(\oB{X_i}{r}{\manifold{M}}) \leq 3/2 \mu(\oB{X_i}{r}{\manifold{M}}) + 3\delta^2\] and
	\[\mu(\oB{X_i}{r}{\manifold{M}}) \leq 3/2 \eta_1(\oB{X_i}{r}{\manifold{M}}) + 3\delta^2\] hold with 
$\delta^2= 4 n^{-1}\ln\left(\frac{4 n^2}{\pr{2}}\right)$.
\end{cor}
\begin{proof}
	Inequality \eqref{eq::oka-aa} 
	implies
	$\sqrt{\eta_1(\oB{X_i}{r}{\manifold{M}})} \leq \sqrt{\mu(\oB{X_i}{r}{\manifold{M}})} +\delta$ and\\
	$\sqrt{\mu(\oB{X_i}{r}{\manifold{M}})} \leq \sqrt{\eta_1(\oB{X_i}{r}{\manifold{M}})} +\delta$ with  probability at least $1-\pr{2}$.
	Squaring the inequalities and using 
	$xy \leq x^2/2+y^2/2 $ with $x=\sqrt{\mu(\oB{X_i}{r}{\manifold{M}})}$ and $y=\sqrt{2}\delta $ leads to the statement.
\end{proof}

Now we are able to  prove \cref{theo::vdoka}.
\begin{proof}[of \cref{theo::vdoka}]\mbox{}\\
	First note that the doubling condition on $\manifold{M}$ implies for some $s\geq 0$ 
	\begin{align}
	\mu\left( \oB{x}{r}{\manifold{M}}\right)\leq  2^{v \ceil*{s}}   \mu\left( \oB{x}{\frac{r}{2^s}}{\manifold{M}}\right)
	\end{align}
	by applying $\ceil*{s}$- times the  \cref{def::vd} (where $\ceil*{s}=\min\{k\in \GZ: k\geq s\}$).

	Then, for fixed $s\geq 0$, for fixed $\pr{1}$ from  \cref{ass::isomap} and $\pr{2}$ with $\delta^2=4n^{-1}\ln(\frac{4 n^2}{\pr{2}})$
	we can derive for any $ r\in\RZ_+$ the inequality  
	\begin{align*}
	\eta_1(\oB{x}{2r}{G,SP})
	&\leq  \eta_1 \left( \oB{x}{(1-\lambda_1 )^{-1}\varepsilon 2 r }{\manifold{M}}\right ) \\
	& \leq \frac{3}{2}\mu \left( \oB{x}{(1-\lambda_1 )^{-1}\varepsilon 2 r }{\manifold{M}}\right ) +3\delta^2  \\
	&\leq \frac{3}{2} 2^{\ceil*{s}v}\mu\left(\oB{x}{\frac{(1-\lambda_1)^{-1}}{2^s}\varepsilon 2 r}{\manifold{M}}\right)+3\delta^2\\
	&\leq \frac{3}{2}2^{\ceil*{s}v} \left( \frac{3}{2}\eta_1\left(\oB{x}{\frac{(1-\lambda_1)^{-1}}{2^s}\varepsilon  2 r}{\manifold{M}}\right)+3\delta^2\right)+3\delta^2\\
	&\leq  \frac{3}{2} 2^{\ceil*{s}v}\left( \frac{3}{2}\eta_1\left(\oB{x}{\frac{4(1+\lambda_2)(1-\lambda_1)^{-1}}{2^s} 2 r+1}{G,SP}\right)+3\delta^2\right)+3\delta^2\\
	\end{align*}
	which holds with probability at least $1-\pr{1}-\pr{2}$, 
	applying \cref{lem::ubOka} and  \cref{cor::dist} and the doubling condition for $\manifold{M}$.
%

	Now fix some $w>0$ and set $s^*:=w+3+ \log_2 (1+\lambda_2)(1-\lambda_1)^{-1}$. 
	Then for all \[r\geq r(w) =
	\left ( 1-\frac{ 1}{2^{w}} \right)^{-1},\] provided that 
 $\eta_1(\oB{x}{r}{G,SP}) \geq  8n^{-1}\ln(\frac{4 n^2}{\pr{2}})$,  we can finally conclude that
	\begin{align*}
	\eta_1(\oB{x}{2r}{G,SP}) \leq 6 \cdot 2^{\ceil*{s^*} v } \eta_1 (\oB{x}{r}{G,SP})=2^{\log_2(6)+\ceil{s^*}v} \eta_1 (\oB{x}{r}{G,SP}).
	\end{align*}  
	%
	%
	For the case 
$r\in \left(1,2\right]$ we have  
	\begin{align*}
		\eta_1(\oB{x}{2r}{G,SP}) &\leq	\eta_1(\oB{x}{4}{G,SP}) \\
		&\leq 2^{\log_2(6)+\ceil{s^*}v} \eta_1 (\oB{x}{2}{G,SP})\\
		& = 2^{\log_2(6)+\ceil{s^*}v} \eta_1 (\oB{x}{r}{G,SP}).
	\end{align*}
\end{proof}

\begin{proof} [of \cref{cor::vd-emp}]
	We set $\pr{1}=\pr{2}:=1/n^z$ with $z>0$.
	Then obviously $1-\pr{1}-\pr{2}=1-2\frac{1}{n^z}$ converges to 1 if $n \rightarrow \infty$. 
	 For $r>1$ we have $\eta_1 \left(B_{G,SP}(X_i,r)\right)\geq \eta_1 \left(\overline{B}_{G,SP}(X_i,1)\right)=n^{-1} \grad{X_i}$.
	We will show that with probability going to 1  \begin{align}\label{cond}
	\grad{X_i}\geq 8 \ln\left(\frac{3n^2}{\pr{2}}\right)=8\ln\left(\frac{3n^2}{ n^{-z}}\right)=8 \ln\left(3n^{2+z}\right)\end{align}
	holds for all center points $X_i$.
	By \cref{rem:deg} we know that for $\delta\in(0,1)$ $\grad{X_i}\geq \grad_{min}\geq (1-\delta)c_l/2 n\varepsilon^k$ with probability at least $1-n \exp(\delta^2 nc_l \varepsilon^k /6)$. In our standard asymptotic regime this probability tends to 1 
and we also have $n\varepsilon^k\geq C \ln(n)$ for every $C$ for $n$ large enough. Putting this together  we obtain $\grad{X_i}\geq (1-\delta) c_l/2 C \ln(n) \geq 8 \ln\left(3n^{2+z}\right) $ with prob at least $1-n\exp(\delta^2 n \varepsilon^k /6)$ for $n$ and $C$  large enough. 
Thus \cref{cond} is satisfied for all balls with probability tending to 1 if $n$ is large enough. 
Moreover, by \cref{rem:A3} \cref{ass::isomap} is satisfied in the standard asymptotic regime for $n$ large enough with probability going to 1.
\end{proof}

Furthermore the doubling property of the empirical graph measure implies a doubling property for the degree volume graph measure (if the vertex degrees are bounded).
\begin{theorem}\label{theo::empDeg}
	Let $G=(\mathcal{V},\mathcal{E})$ be a fixed graph. 
	Let $\eta_1$ be the empirical graph measure and $\eta_2$ the degree volume graph measure.
	
	If $(G,\eta_1)$ satisfies (rDV[$v, r_-, r_+$]),
	then  $(G,\eta_2)$ satisfies (rVD[$\tilde{u},r_-,r_+$])
	with \linebreak$\tilde{u}=u+2\log_2\left( \frac{ \max_{x \in \mathcal{V}}\grad(x)}{\min_{x\in \mathcal{V}}\grad(x)} \right)$.
\end{theorem}

\begin{proof} Let $B\subset \mathcal{V}$ and $n_B$ the number vertices in $B$.
	Then we can bound $\vol(B)$:
	\[n_B \min_{x\in \mathcal{V}} \grad(x) \leq \vol(B) \leq n_B \max_{x\in \mathcal{V}} \grad(x).\]
	Therefore the degree volume graph measure can be bounded in terms of $\eta_1$ by
	\[c_{\bullet}^{-1}  \eta_1(B) \leq  \frac{\vol(B)}{\vol(\mathcal{V})}=\eta_2(B) \leq c_{\bullet}  \eta_1(B)\]
	with
	$c_{\bullet}=\frac{\max_{x\in \mathcal{V}}\grad(x)}{ \min_{x \in \mathcal{V}}\grad(x)}\geq 1$. 
	Then we get immediately, assuming that the graph satisfies the doubling property with doubling constant $u$ that
	\begin{align*}
	\eta_2(B_{G,SP}(X_i,2r))&\leq c_{\bullet} \eta_1(B_{G,SP}(X_i,2r)) 
	\stackrel{\mathcal{V}D(u)}{\leq} c_{\bullet} 2^u \eta_1(B_{G,SP}(X_i,r)) \\
	&\leq 2^u c_{\bullet}^2 	\eta_2(B_{G,SP}(X_i,r))
	\end{align*}
\end{proof}

Now  \cref{vd-deg} follows directly from \cref{theo::vdoka} and \cref{theo::empDeg}.

\begin{proof} of \cref{cor::vd-deg}
	We set $\pr{1}=\pr{2}:=n^{-z}$ for some $z>0$.
	As shown in the proof of \cref{cor::vd-deg} the  condition $ \eta_1\left(\oB{X_i}{r}{G,SP}\right)\geq 8 n^{-1}\ln\left(\frac{3n^2}{\pr{2}}\right)$ holds with probability tending to 1 for all balls. Furthermore, based on the Ahlfors regularity and \cref{rem:deg} we have $c_{\bullet}:=\frac{(1+\delta)c_u }{(1-\delta)(1-\lambda_1)^k c_l}\geq \frac{\max_{x\mathcal{V} \grad(x)}}{\min_z\in \mathcal{V}\grad(z)} $ with probability at least $1-\pr{3}$ with $\pr{3}=n\exp{-\delta^2 n c_\varepsilon^k /6}$.
	In our standard asymptotic regime $\pr{3}$ converges to 0 and
	$1-\pr{1}-\pr{2}-\pr{3}$ converge to 1 for $n\rightarrow \infty$. 
	Moreover, by \cref{rem:A3} \cref{ass::isomap} is satisfied in the standard asymptotic regime for $n$ large enough with probability going to 1. 
\end{proof}

It remains to prove  \cref{theo::ubOka-a}.
Let us recall the following classical Okamoto inequality (see e.g. \cite[Theorems 3+4]{okamoto59}) which is needed to bound the difference of true (manifold) and empirical (graph) measure uniformly over all balls.

\begin{lemma} [Okamoto's inequality] \label{lem::OI}
Let $Y_i\sim \BinV (p)$ i.i.d. with  $\ew{ Y_i}= p\in[0,1]$ and set  $\hat{p}:=\frac{1}{m}\sum_{i=1}^{m}Y_i$. 
 Then, for $\delta>0$, 
\begin{align*}
\prob{\sqrt{\hat{p}}\geq \sqrt{p}+\delta}\leq \exp(-2m \delta^2), \quad
\prob{\sqrt{p}\geq \sqrt{\hat{p}}+\delta}\leq \exp(-m \delta^2).
\end{align*}
\end{lemma}

\begin{proof}[of \cref{theo::ubOka-a}]\mbox{}
	We want to prove
	\begin{align}\label{ineq::2}
	\prob{\sup_{i=1..n}\sup_{r> 0} \abs{\sqrt{\eta_1(\oB{X_i}{r}{\manifold{M}})}-\sqrt{\mu(\oB{X_i}{r}{\manifold{M}})}} > \delta} \leq \pr{2}.
	\end{align}
	To shorten notation we define
	\begin{align*}
	T_{i,r}&:=\sqrt{\eta_1(\oB{X_i}{r}{\manifold{M}})}-\sqrt{\mu(\oB{X_i}{r}{\manifold{M}})},\\
	\overline{T}_{i,r}&:=\sqrt{\eta_1(\cB{X_i}{r}{\manifold{M}})}-\sqrt{\mu(\cB{X_i}{r}{\manifold{M}})}.
	\end{align*}
	We will first bound the left-hand-side of \cref{ineq::2} by using  
	the union bound and conditioning  on the center points of the balls:
	\begin{align*}
	&\prob{  \sup_{i=1..n} \sup_{r>0} \abs{T_{ir}} > \delta  } 
	&\leq \sum_{i=1}^{n} \prob{ \sup_{r> 0} \abs{T_{ir}}> \delta  } 
	&= \sum_{i=1}^{n} \ewt[X_i]{\prob{ \sup_{r> 0} \abs{T_{ir}} > \delta \given X_i }}  
	\end{align*}
	
	As second step we establish a upper bound for $  \sup_{r> 0} \abs{T_{i,r}}$ for fixed $i$.  
	Without loss of generality we take $i=1$ and we will abbreviate $T_r=T_{1,r}$ and $\overline{T}_r=\overline{T}_{1,r}$ .
	Then we define $r_j=\distb{\manifold{M}}(X_1,X_j) , r_{n+1}=\infty$ and denote  $\{r_{(j)}\}_{1\leq j\leq n+1} $ the reordered values of $\{r_{j}\}_{1\leq j\leq n+1}$ with  $0=r_{(1)}\leq r_{(2)}\leq \ldots\leq r_{(n)}< r_{(n+1)}=\infty$. 
	
		Then $\sup_{r> 0} \abs{T_{r}}\leq  \max \{E_{1},E_{2},E_{3}\}=\max\{E_{1},E_{2}\}$ with 
			\begin{align*}
		E_{1}&:=\max_{1\leq j \leq n+1}\abs{T_{ r_{(j)}}}  =\max_{ 1\leq j \leq n+1}\abs{T_{ r_{j}}}\\
		E_{2}&:=\max_{1\leq j\leq n} \overline{T}_{r_{(j)}}   = \max_{1\leq j\leq n} \overline{T}_{r_{j}}\\
		\text{and~}  E_{3}&:=\max_{1\leq j\leq n} -T_{r_{(j+1)}} =\max_{2\leq j\leq n+1 } -T_{ r_{(j)}}=\max_{2\leq j \leq n+1} -T_{ r_{j}}.
		\end{align*}
		
	To achieve this we decompose 
	the set 	$ \{ r>0 \}=R_1\cup R_2$ with $R_1=\{ r_{(j)}: r_{(j)}\neq 0 , j \leq n\}$. The supremum of $\abs{T_{r}}$ over $r\in R_1$ can then obviously bounded by $E_{1}$. 
	
	For $r\in R_2$ we have  $r\in \left( r_{(j)}, r_{(j+1)} \right)$ for some $j$ and  we can bound $\abs{T_{r}}$ from above by   $\max\{ \overline{T}_{r_{(j)}}, -T_{r_{(j+1)}}\}$ exploiting that  a) $\eta_1(\cB{X_1}{r}{\manifold{M}})$ ~is~ constant~ for~ $r_{(j)} \leq r < r_{(j+1)}$, \linebreak
b) 	$\eta_1(\oB{X_1}{r}{\manifold{M}})$~is~ constant~ for~ $r_{(j)} < r \leq r_{(j+1)}$ and
	especially \linebreak c) $\eta_1\left( \oB{X_1}{r_{(j+1)}}{\manifold{M}}\right)=\eta_1\left( \cB{X_1}{r_{(j)}}{\manifold{M}}\right)$ and d) $\mu(\oB{X_1}{r}{\manifold{M}})$ is increasing in $r$:  for $r\in \left( r_{(j)}, r_{(j+1)} \right)$
		\begin{align*}
	T_{r}&\leq \sqrt{\eta_1\left(\cB{X_1}{r_{(j)}}{\manifold{M}}\right) }-\sqrt{\mu\left(\cB{X_1}{r_{(j)}}{\manifold{M}}\right) }&=& \overline{T}_{r_{(j)}}\\
	-T_{r}&\leq \sqrt{\mu\left(\oB{X_1}{r_{(j+1)}}{\manifold{M}}\right) }- \sqrt{\eta_1\left(\oB{X_1}{r_{(j+1)}}{\manifold{M}}\right) }&=&-T_{r_{(j+1)}}.
	\end{align*}
	
	This leads to the choice of $E_{2}$ and $E_{3}$. Note that $E_3\leq E_1$.	
	We can now write 
	\begin{align*}
	\ewt[X_1]{\prob{ \sup_{r> 0} \abs{T_{r}} > \delta \given X_1 }} 
	&\leq  \ewt[X_1]{ \prob{ \max \{E_{1},E_{2} \} > \delta \given X_1 }}  
	\\
	&\leq \ewt[X_1]{   \sum_{k=1}^{2}\prob{E_{k}> \delta \given X_1 }}.  
	\end{align*}
	%
	Note that the random variables $E_{1}$ and $E_{2}$ can be written in terms of the ordered radii $r_{(j)}$ or in terms of the unordered radii $r_{j}$ which is suitable for the further computations.
	
	As third step,  in order to bound the probabilities $\prob{E_{k}> \delta \given X_1 }$,   we need upper bounds for 
	$\prob{\abs{T_{ r_{j}}}>\delta|X_1}$ and $ \prob{\overline{T}_{ r_{j}}>\delta|X_1}$. 
	
	Our approach is to  additionally condition on $X_j$ 
	and apply Okamoto's inequality (see \cref{lem::OI}).
	Conditionally to   $X_1$ and $X_j$,  the ball $B_j=\oB{X_1}{r_{j}}{\manifold{M}}$ has a fixed center point and a fixed radius. Then $\mu(B_j)$ is a number, $\eta(B)$ is still random, depending on the other $n-2$ (respectively $n-1$ in the special case when $r_{11}$ is used) points.
	 Note that $\eta_1(B_j)$ conditioned on $X_1$ and $X_j$ is biased.
	Therefore we need an adjusted measure $\hat{\eta}_1$ which is (conditioned on $X_1,X_j$) unbiased for $\mu$. 
	First we control the deviation of $T_{r_j}$. We can assume $j>1$ and $j\leq n$ since $T_{r_1}=T_{r_{n+1}}=0$.
	Let's consider $\prob{\abs{T_{ r_{j}}}>\delta|X_1}$ with $0<r_{j}<\infty$.
	We define the random variable $\hat{\eta}_1$ as
	\begin{align*}
	\hat{\eta}_1 (B_j) := \frac{1}{n-2}\sum_{\substack{k=2,\ldots,n\\  k\neq j}} \Indik(X_k\in B_j )
	\end{align*}
	Then conditionally on $X_1$ and $X_j$,  $(n-2)\hat{\eta}_1(B_j)$ is binomial-distributed with parameters $n-2$ and $\mu (B_j )$.
	Note that the equality
$	\hat{\eta}_1 (B_j) (n-2) +1= n  \eta_1(B_j)$
	holds and it implies \begin{align}\label{ineq::mu}
	\abs{\hat{\eta}_1 (B_j)-\eta_1(B_j)}\leq 1/n.
	\end{align}
Using that $\abs{\sqrt{a}-\sqrt{b}}\leq \sqrt{\abs{a-b}}$  and \cref{ineq::mu} we get 
	\begin{align*}
	T_{r_{j}} 
	&\leq  \sqrt{\frac{1}{n} }+\abs{ \sqrt{\hat{\eta}_1 (B_j)}- \sqrt{\mu (B_j)} }.
	\end{align*}
	Then we obtain for $\delta-\frac{1}{\sqrt{n}}>0$ by applying Okamoto's inequality (see \cref{lem::OI})
	\begin{align*}
	\prob{\abs{T_{ r_{j}}} >\delta \given  X_1, X_j} 
	&\leq \prob{\abs{\sqrt{\hat{\eta}_1(B_j)}-\sqrt{\mu (B_j)}}> \delta -\frac{1}{\sqrt{n}} \given X_1,X_j} \\
	& \leq 2 \exp\left( - (n-2)\left(\delta -\frac{1}{\sqrt{n}}\right)^2\right)
	\end{align*}
	and by the union bound we get 
	\begin{align*}
	\prob{E_{1}>\delta  \given X_1} 
	&\leq \sum_{2\leq j\leq n}  \ewt [X_j]{\prob{ \abs{T_{ r_{j}}} > \delta \given X_1,X_j}}\\
	& \leq 2 (n-2) \exp \left( - \left( \delta-\frac{1}{\sqrt{n}}\right)^2\left( n-2\right)\right).
	\end{align*}
	Second, for the deviation of $\abs{T_{r_j}}$, that is for
	$\prob{E_{2} >\delta \given X_1}$ 
	we get the following similar result adapting the way of computation. 
	We consider closed balls with $r_{j}>0$.
	Note that for closed balls we get $2/n$ as bound instead of $1/n$ in inequality \eqref{ineq::mu} (since the closed ball includes one additional point).
	
	\begin{align*}
	\prob{E_{2} >\delta \given X_1} 
	&\leq \sum_{j: r_{j}\neq 0}  \ewt [X_j]{\prob{ \overline{T}_{ r_{j}} > \delta |X_1,X_j}}\\
	& \leq n \exp\left(- \left( \delta-\sqrt{\frac{2}{n}}\right)^2\left( n-2\right)\right).
	\end{align*}

	Finally we get (by plugging in the upper bounds)
	\begin{align*}
	\prob{\sup_{i} \sup_{r>0}  \abs{\sqrt{\eta_1(\oB{X_i}{r}{\manifold{M}})}-\sqrt{\mu (\oB{X_i}{r}{\manifold{M}})}}> \delta}
	\end{align*}
	\begin{align*}
	&=\prob{\sup_i \sup_{r>0} \abs{T_{i r}} > \delta } \\
	&\leq \sum_{i=1}^{n}  \ewt[X_i]{   \sum_{k=1}^{2}\prob{E_{k,i}> \delta \given X_i }} \\
	& \leq 3 n^2 \exp\left( -(n-2) \left( \delta -\sqrt{ \frac{2}{n}}\right)^2 \right). \\
	\end{align*}
	To finish the proof we choose $\delta$ such that  $\pr{2}\geq 3 n^2 \exp\left( -(n-2) \left( \delta -\sqrt{ \frac{2}{n}}\right)^2 \right)$ is satisfied for fixed $\pr{2}\in(0,0.5]$.
	We obtain $\delta:=2\sqrt{\frac{\ln\left(\frac{3n^2}{\pr{2}}\right)}{n}}$ since for $\pr{2}\in(0,0.5]$ and $n\geq 4$ we have $\sqrt{\frac{\ln(\frac{\pr{2}}{3n^2})}{n-2}}+\sqrt{\frac{2}{n}}\leq 2\sqrt{\frac{\ln\left(\frac{3n^2}{\pr{2}}\right)}{n}}$.
	\end{proof}

\begin{remark}
	\cref{lem::ubOka} holds true for closed balls. Only minor adjustments in the proof are necessary.
\end{remark}

\section{Local Poincar\'{e} inequality}\label{Poin}

\subsection{ \cref{theo:LPIsp}}
Our results (\cref{theo::LPIspPi2,cor::lpi-emp}) comprise the local Poincar\'{e} inequality for the empirical and the degree volume graph measure. These results follow from the more general result  \cref{theo:LPIsp} which applies to a general probability measure $\eta$ satisfying the following assumption.

\begin{ass2}
	\begin{assenumE}
		\item \label{ass::poseta} 	Let $G=(\mathcal{V}, \mathcal{E})$ be a graph with $\abs{\mathcal{V}}=n<\infty$ and $\eta$ a discrete probability measure defined on $\mathcal{V}$ satisfying \[\eta(x)>0 \text{~for~ all~} x\in\mathcal{V}\]
		and denote 
		$\eta^{+}:=\max_{x\in\mathcal{V}} \eta(x)$ and $\eta^{-}:=\min_{x\in\mathcal{V}} \eta(x)$.
	\end{assenumE}
\end{ass2}
Note that under \cref{ass::poseta} we have $0< \eta^-\leq \eta^+<\infty$. 

\nomenclature[Cz]{$n_A$}{number of points in $A\cap \mathcal{V}$}
\nomenclature[Cz]{$\eta^{+}, \eta^{-}$}{maximal and minimal value of $\eta$ }

Essential for \cref{theo:LPIsp} is the existence of a certain bi-Lipschitz homeomorphism. 
\begin{defi}\label{bilipDef}
	We call $h:\mathcal{X}\rightarrow \mathcal{Y}$, $\mathcal{X}, \mathcal{Y}$ compact metric spaces, a bi-Lipschitz homeomorphism if 
	$h$ is a homeomorphism (bijective, continuous, and the inverse $h^{-1}$ is continuous) and there exist constants $0<L_{min}<L_{max}<\infty $ such that for all $x,y\in \mathcal{X}$ 
	\begin{align}
	L_{min}  \norm{x-y}_{\mathcal{X}} \leq \norm{h(x)-h(y)}_{\mathcal{Y}}\leq L_{max} \norm{x-y}_{\mathcal{X}}.\label{bilip}
	\end{align}
\end{defi}
In this work, $\mathcal{X}$ will be a closed ball $\overline{B}_{\manifold{M}}(x,r)$ and $\mathcal{Y}=[0,1]^k$. 
In this case it is natural to assume that $L_{min}:=\frac{L_{min}^*}{r}$ and $L_{max}:=\frac{L_{max}^*}{r}$ with constants $L_{min}^*$ and $L_{max}^*$ independent of $r$ and $0<L_{min}^*<L_{max}^*<\infty$. 
We introduce the following assumption on the manifold (we will see that it will be implied by \cref{ass::secCurv}).
\begin{ass2}
	\begin{assenumS}
		\item \label{ass::lip}
		There exist universal constants $0<L_{min}^*<L_{max}^*<\infty$ 
		such that for all \linebreak$0<r_{\manifold{M}}<r_{max}$ and for all $x_0\in \manifold{M}$ there exists a bi-Lipschitz homeomorphism $h: \overline{B}_{ \manifold{M}}(x_0,r_{\manifold{M}})\rightarrow  [0,1]^k$ satisfying
			\begin{align}
		\frac{L_{min}^*}{r_{\manifold{M}}}  \norm{x-y} \leq \norm{h(x)-h(y)}\leq 	\frac{L_{max}^*}{r_{\manifold{M}}} \norm{x-y}
		\end{align}
		for all $x,y\in \overline{B}_{ \manifold{M}}(x_0,r_{\manifold{M}})$.
	\end{assenumS}
\end{ass2}

\Cref{ass::lip} is a technical condition on the underlying manifold needed for (LPI) which allows us to control
number and length of paths from a specific path class.  
It implies that for a ball $\overline{B}_{\manifold{M}}(X_i,r)$ with random center point $X_i\in \mathcal{V}$ the function $h$ is bi-Lipschitz almost sure.

Under these assumptions we state the following local Poincar\'{e} inequality in $d_{SP}$-distance.

\begin{theorem}[LPI]\label{theo:LPIsp} 
	Let $G=(\mathcal{V}, \mathcal{E})$ be an $\varepsilon$-graph defined from an i.i.d. sample of size $n$ from  the probability  measure $\mu$ on the submanifold $\manifold{M}$ of $\RZ^K$ such that \cref{ass::manifold,ass::graph,ass::isomap,ass::poseta,ass::Ahlfors,ass::lip} 
	are satisfied (with parameters  $\lambda_1, \lambda_2, \pr{1}, \varepsilon>0, n\geq 2, c_l, c_u, k, L_{min}^*, L_{max}^*, r_{max}$). 	Let  $\delta\in (0,1)$.
	Assume $r_{max}(1-\lambda_1)\varepsilon^{-1}\geq 1$, 	$n\geq \frac{1}{(1-\delta)c_l} \left(\frac{4 \sqrt{k+3} L_{max}^*}{L_{min}^* \varepsilon}\right)^k +1$ 
	and $\frac{\sqrt{k+3}}{L_{min}\varepsilon}\geq 1$,  
	and define
	\[\pr{4}:= 2	\left(\frac{2 \sqrt{k+3}n (1-\lambda_1)}{L_{min}^*}\right)^k  \exp\left(-\frac{\delta^2 n c_l}{6} \frac{\varepsilon^k {L_{min}^*}^k}{4^k \sqrt{k+3}^k {L_{max}^*}^k}\right) +  2 \exp\left(-\frac{\delta^2 n c_l \varepsilon^k }{6(1-\lambda)^{k}}\right).\]

	Then  there exist  constants $\lambda>0$ and $C_{*}>0$ such that, with probability at least 
	$1-n^2\pr{4}-\pr{1}$,
	for all balls $B=\overline{B}_{SP}(X_i,r)$ 
	with  $r\in [1,\min(r_{max}(1-\lambda_1)\varepsilon^{-1},n))$  
	and $X_i\in \mathcal{V}$,   and for all functions $f: \mathcal{V}\rightarrow \RZ$ 
	the inequality
	\begin{align}
	\sum_{x\in \overline{B}_{SP}(X_i,r)} (f(x)-\overline{f}_{\overline{B}})^2 \eta(x) 
	&\leq C_{*}  r^2 \sum_{\substack{x,y \in \overline{B}_{SP}(X_i,\lambda r)\\ x\sim y }}  (f(x)-f(y))^2 \label{lpi}
	\end{align}
	holds.\\
	The constants are  $\lambda=4\frac{(1+\lambda_2)}{(1-\lambda_1) }+1$ and 
	\begin{align}
	C_{*} &= 
	\frac{1}{\varepsilon^k n}\frac{1 }{(1-\lambda_1)^2}   \frac{(\eta^+)^2}{(1-\delta)c_l \eta^-} 	\left(1+w\right)^2  k^2 \left(2\frac{\sqrt{k+3}}{L_{min}^*}\right)^{k+2}
	\end{align}
	
	with $w:=\frac{2(1+\delta)}{1-\delta}\frac{c_u}{c_l}\frac{{L_{max}^*}^k}{{L_{min}^*}^k}4^k \sqrt{k+3}^k.$ 	
\end{theorem}
We prove \cref{theo:LPIsp} in \cref{sec::gen}. 

At this point we show that \cref{theo::LPIspPi2,cor::lpi-emp} are consequences of \cref{theo:LPIsp} applied to the specific graph measures $\eta_1$ and $\eta_2$,  respectively. 
We will see that these measures satisfy \cref{ass::poseta} at least with high probability 
and that the existence of the  bi-Lipschitz homeomorphism  (\cref{ass::lip}) is guaranteed by \cref{ass::secCurv}.   

\begin{lem}[Existence of bi-Lipschitz homeomorphism]\label{existbilip}
	Let $\manifold{M}$ be a $k$-dimensional submanifold of $\RZ^K$ such that \cref{ass::manifold,ass::secCurv} are satisfied with parameters $i(\manifold{M})$ and $\Lambda$.
	
	Then \cref{ass::lip} holds with $r_{max}:=r_{\bullet}=\min \left(\frac{i(\manifold{M})}{2},\frac{\pi}{2\sqrt{\Lambda}}\right)$.
\end{lem}
The proof can be found at the end of this section. 
Now we are able to prove our main results.

\begin{proof}[of \cref{cor::lpi-emp}]
	Obviously, $\eta_1$ satisfies \cref{ass::poseta} and $\eta^+=\eta^-=1/n$.
	\Cref{ass::lip} is satisfied by \cref{existbilip} with $r_{max}:=\min \left(\frac{i(\manifold{M})}{2},\frac{\pi}{2\sqrt{\Lambda}}\right)$.
	Now we apply \cref{theo:LPIsp} and multiply each side of \cref{theo:LPIsp}	with $1/n$ and consider therefore $n\cdot C_{*}$. The quantity dependent on $n,\varepsilon$ reduces to
	\begin{align}n\frac{ (\eta_1^+)^2}{\eta_1^-} \frac{1}{n\varepsilon^k}&= \frac{1}{n\varepsilon^k}.
	\end{align}
	The constant is therefore 
	\begin{align*}
	\hat{C}:=n C_*=&const\left( \lambda_1, c_l, L_{min}^*, k, w, \delta \right) \cdot \frac{1}{n\varepsilon^k}\\
	=& 
	\frac{n}{\varepsilon^k n}\frac{1 }{(1-\lambda_1)^2}   \frac{(\eta^+)^2}{(1-\delta)c_l \eta^-} 	\left(1+w\right)^2 k^2 \left(2\frac{\sqrt{k+3}}{L_{min}^*}\right)^{k+2}.\end{align*}
	We set $\pr{1}:=n^{-z}$ for $z>0$ and observe that under our standard asymptotics $n^2\cdot\pr{4}$  from \cref{theo:LPIsp} converge to 0 and $\pr{1}\rightarrow 0$ for $n$ going to infinity.
	Moreover, by \cref{rem:A3} \cref{ass::isomap} is satisfied in the standard asymptotic regime for $n$ large enough with probability going to 1. 
\end{proof}

\begin{proof} [of  \cref{theo::LPIspPi2}]
	\Cref{ass::lip} is satisfied by \cref{existbilip} with \linebreak $r_{max}:=r_{\bullet}=\min\left(\frac{i(\manifold{M})}{2},\frac{\pi}{2\sqrt{\Lambda}}\right)$. 
	Recall that the degree volume graph measure is given by
	\[\eta_2(x)= \frac{\grad(x)}{\vol(\mathcal{V})}\quad  \forall x\in \mathcal{V}.\]
	%
	%
	By applying \cref{theo::edegrees,rem:deg} we get that with probability at least \linebreak $1-2n\exp(-\delta^2(n-1)c_l \varepsilon^k/3)$ 
	\begin{align*}
	\grad_{min}&> (1-\delta)(n-1)c_l \varepsilon^k >0 \text{~and~}\\
	\grad_{max}&< (1+\delta)(n-1) c_u (1-\lambda_1)^{-k}\varepsilon^k <\infty.\end{align*}
	This implies that  with probability at least $1-2n\exp(-\delta^2(n-1)c_l \varepsilon^k/3)$ \cref{ass::poseta} holds. Now we can apply \cref{theo:LPIsp} and multiply each side with $\vol(\mathcal{V})$. We consider the term $\vol(\mathcal{V})C_{*}$. The quantity depending on $\varepsilon$ and $n$ can be bounded with probability at least $1-2n\exp(-\delta^2(n-1)c_l \varepsilon^k/3)$
	\begin{align}
	\vol{\mathcal{V}}\frac{(\eta_2^+)^2(x)}{\eta_2^-(y)} \frac{1}{n\varepsilon^k}&= \frac{\max_x \grad^2(x)}{\min_y \grad(y) n\varepsilon^k}\\
	&\leq \frac{(1+\delta)^2(n-1)^2 c_u^2 (1-\lambda_1)^{-2d}\varepsilon^2d}{(1-\delta)(n-1)c_l \varepsilon^k n \varepsilon^k}\\
	&\leq \frac{(1+\delta)^2 c_u^2}{(1-\lambda_1)^{2d}(1-\delta)c_l}.
	\end{align}
	The constant is \[\hat{C}:=\vol{\mathcal{V}\cdot C_*}
	\frac{(1+\delta)^2c_u^2}{(1-\lambda_1)^{2(k+1)}(1-\delta)^2 c_l^2 } 	\left(1+w\right)^2 k^2 \left(2\frac{\sqrt{k+3}}{L_{min}^*}\right)^{k+2}.\]	
\end{proof}

\begin{proof}[of \cref{existbilip}]
	Let us consider $\overline{B}_{ \manifold{M}}(x,r)$ for fixed $x\in \manifold{M}$ and fixed $0<r\leq r_{max}$. 
	We construct $h=g_0 \circ g_1 \circ g_2 \circ g_3: \overline{B}_M(x,r)\rightarrow [0,1]^k $ (with $h(x)=g_0(g_1(g_2(g_3(x))))$) as a composition of four bi-Lipschitz homeomorphisms.\\
	We define $g_2:\cB{0}{r}{E} \subset \RZ^k \rightarrow   \cB{0}{1}{E} \subset\RZ^k, x\mapsto x/r$.
	This function is continuous and bijective. The inverse function $h^{-1}(y)=yr$ is also continuous. Since $\norm{\frac{x}{r} -\frac{z}{r}}= 1/r \norm{x-z }$, $h$ is a bi-Lipschitz homeomorphism satisfying \cref{bilip} with $L_{g_2,min}=L_{g_2,max}=\frac{1}{r}$.\\
	The existence of the bi-Lipschitz homeomorphism $g_1: \cB{0}{1}{E}\subset \RZ^k \rightarrow [-1,1]^k$ is proved by \cite[Cor. 3]{WIAS}. The Lipschitz constants $L_{g_1,min}, L_{g1,max}$ do not depend by definition on the properties of the submanifold or the radius $r$.\\
	We define $g_0:[-1,1]^k \subset \RZ^k \rightarrow  [0,1]^k, x\mapsto \frac{x+1}{2}$ which is as $g_2$ a bi-Lipschitz homeomorphism with Lipschitz constants $L_{g_0,min}=L_{g_0,max}=0.5$.\\
	Finally we set
	$g_3:\cB{x}{r}{\manifold{M}} \rightarrow \cB{0}{r}{E} \subset \RZ^k$
	to be the inverse of the exponential map  $\exp_x: T_x\manifold{M} \rightarrow \manifold{M}$ restricted to the domain $\cB{0}{r}{E}\subset T_x\manifold{M}$. It is known that the exponential map is a diffeomorphism when the domain is restricted to a ball of radius smaller than the injectivity radius of $\manifold{M}$. Furthermore the ball $\overline{B}_{\manifold{M}}(p,r)$ is (strongly) convex if $r\leq \min(i/2, \frac{\pi}{2\sqrt{\Lambda}})$ (see \cite{chavel}[Theorem IX.6.1] or \cite{Rauch-Buser}[Prop. 6.4.6]).
	The Rauch Theorem stated in \cite[Lemma 5]{RauchTheorem} provides us with bounds on the derivative of $\exp_p$. 
	Considering the length of the image $\exp_p(s(t))$ of the linear segment $s(t)$ connecting $v$ to $v'$ ($v,v'\in \overline{B}_E(0,r)$) and the length of the image $\exp_p^{-1}(u(t))$ of the geodesic $u(t)$ connecting $y=\exp_p(v)$ to $y'=\exp_p(v')$ in $\overline{B}_{\manifold{M}}(p,r)$ we get by  the chain rule, the inverse function theorem on manifolds and the convexity of the ball the following Lipschitz constants:
	$L_{g_3,min}:=\left(1+\frac{\pi^2}{8}\right)^{-1}\geq 0.4$ 
	and
	$L_{g_3,max}:=\left(1-\frac{\pi^2}{24}\right)^{-1}\leq 1.7$. 
	
	
	%
	%
	
	Since the composition of Lipschitz function is Lipschitz, we obtain
	$L_{min}\norm{x-y}\leq \norm{h(x)-h(y)}\leq L_{max}\norm{x-y}$
	with  $L_{max}:=L_{g_0,max} L_{g_1,max} L_{g_2,max}  L_{g_3 ,max} = \frac{L_{max}^*}{r}$ and $L_{min}:= L_{g_0,min} L_{g_1,min} L_{g_2,min} L_{g_3 ,min}=\frac{L_{min}^*}{r}$ with $L_{min}^*,L_{max}^*\in\RZ$.
	Thus we constructed a bi-Lipschitz homeomorphism for the fixed ball with  constants $L_{min}^*, L_{max*}^*$ independent of $x$ and $r$ which is possible for any ball $\overline{B}_{\manifold{M}}(x,r)$ with $x\in\manifold{M}$ and $0<r<r_{max}$.  
\end{proof}

\subsection{Proof of \cref{theo:LPIsp} }\label{sec::gen}
We first concentrate on proving the local Poincar\'e inequality for a given ball. \Cref{theo:LPIsp} 
will then be obtained by a union bound over the center points and radii which are of finite number.

The four main ingredients of the proof of \cref{theo:LPIsp} 
are a general approach of \citeauthor{Diaconis} (\cite{Diaconis}) to derive a Poincar\'{e}-type  inequality involving the quantities maximal average load and maximal path length,  the random Hamming paths ( introduced in \cite{Ulrike2014}) as tool for bounding the maximal average load and the maximal path length,  Ahlfors regularity and the distance approximation ($d_{SP} \approx \distb{\manifold{M}}$, see \cref{sec::distapp}).

Let us shortly explain our way to prove \cref{theo:LPIsp}.
We first recall a result on the general structure of the inequality including the quantity $\kappa$ and  we present an upper bound for $\kappa$. Then we derive a local Poincar\'e inequality  first in $\distb{\manifold{M}}$ and finally in $d_{SP}$.

We start by deriving a local Poincar\'e inequality for balls in $\distb{\manifold{M}}$-distance. We obtain the general structure of the Poincar\'e inequality by an approach of \cite{Diaconis}. 

To this end we define $\Gamma$ to be a collection of paths, consisting of one path connecting $x$ to $y$ in $G$ for every pair of points $x,y\in \mathcal{V}$. Moreover, 
we define the maximal path length of the collection $l_{max}(\Gamma):=\max_{\gamma\in \Gamma} \NE{\gamma}$ and the load of an edge $b(e,\Gamma):=\sum_{\gamma\in \Gamma: \gamma\ni e} 1$.

\begin{theorem}[general structure]\label{genStruc}
	Let $G=(\mathcal{V}, \mathcal{E})$ be a given connected graph. 
	Let $\eta$ be a discrete probability measure on $\mathcal{V}$. 
	Let $\Gamma$ be a collection of paths.
		
	Then there exists a quantity $\kappa:=\kappa(\Gamma, \eta)>0$ such that for all functions $f: \mathcal{V}\rightarrow \RZ$ 
	the inequality 
	\begin{align}\label{gs-2}
	\sum_{x\in \mathcal{V}} (f_x-\overline{f})^2 \eta(x)  \quad 
	&\leq  \kappa 
	\sum_{x\in \mathcal{V}}\sum_{\substack{y\in \mathcal{V},\\ y\sim x}} (f_x-f_y)^2 
	\end{align}
	holds where
	$\overline{f}=\sum_{x\in \mathcal{V}} f_x \eta(x)$ 
	and $\kappa:= 0.5  \max_{z\in \mathcal{V}} \eta^2(z) l_{max}(\Gamma) \max_{e\in \mathcal{E}}b(e,\Gamma)$. 
\end{theorem}

\begin{proof}
	For the proof we will make use of the ideas of \citeauthor{Diaconis} presented in \cite{Diaconis}. 
	We start with fixing a path $\gamma_{xy}$ in $G$ for every pair of points $x,y\in \mathcal{V}$. Let $\Gamma:=\{\gamma_{xy}:x,y\in \mathcal{V}\}$.
	Then it is well-known that
	\begin{align*}
	\sum_{x\in \mathcal{V}} (f_x-\overline{f})^2 \eta(x)
	&=\frac{1}{2} \sum_{x\in \mathcal{V}}\sum_{y\in \mathcal{V}} \left(f_x-f_y\right)^2 \eta(x)\eta(y)
	\end{align*}
	holds. Now substitute $f(x)-f(y)$ by $\sum_{e\in \gamma_{xy}}\Delta(e)\left(\frac{g(e)}{g(e)}\right)^{1/2}$
	where $\gamma_{xy}$ is the fixed path from $x$ to $y$ in $G$	and $\Delta(e):=f(a)-f(b)$ if $e=(a,b)$ and $g(e)> 0$ for all $e\in \mathcal{E}$. Denote $Q_{xy}:=\sum_{e\in \gamma_{xy}}\frac{1}{g(e)}$. Then 
	\begin{align}
	\sum_{x\in \mathcal{V}} (f_x-\overline{f})^2 \eta(x) 
	&\stackrel{1)}{\leq} 0.5 	\sum_{x\in \mathcal{V}}\sum_{y\in \mathcal{V}}  \left(\sum_{e\in \gamma_{xy}} \Delta^2(e) g(e)\right) Q_{xy}
	\eta(x)\eta(y)\nonumber \\
	&\stackrel{2)}{=}0.5 \sum_{e\in \mathcal{E}} \Delta^2(e) g(e) \sum_{x,y\in \mathcal{V}: \gamma_{xy}\ni e} Q_{xy}\eta(x)\eta(y)\nonumber\\
	&\stackrel{3)}{\leq}  T \sum_{e\in \mathcal{E}} \Delta^2(e) g(e) b(e, \Gamma) \label{bl} \\ 
	&\stackrel{4)}{\leq}  \kappa \sum_{\substack{x,y\in \mathcal{V},\\x\sim y}}  \left(f(x)-f(y)\right)^2  \nonumber
	\end{align}
	follows by 1) applying Cauchy-Schwarz, 2) rearranging the summation and 3) setting 
	$T:= 0.5  \max_{z\in \mathcal{V}} \eta^2(z) \max_{x,y\in \mathcal{V}} Q_{xy}$ and using the definition of the load. 
	In 4)  we finally set for all edges $g(e)=g(a,b)=a_{ab}$ for $e=(a,b)$. Then $Q_{xy}=\NE{\gamma_{xy}}$. Using  
	the maximal path length $l_{max}(\Gamma)$ 
	we can choose
	\[\kappa:= 0.5  \max_{z\in \mathcal{V}} \eta^2(z) l_{max}(\Gamma) \max_{e\in \mathcal{E}}b(e,\Gamma).\]
\end{proof}

\begin{remark}
	The generalization to a randomly chosen set of paths $\Gamma$ (that is every path $\gamma_{xy}$ is chosen at random from a set of possible paths $\tilde{\Gamma}$) follows \cite{Boyd05}.
	We observe that in \cref{bl} only $\kappa$, in particular $b(e,\Gamma)$, depends on the set of chosen paths. By taking the expectation w.r.t. the randomly chosen set of paths 
	we obtain
	\begin{align*}
	\sum_{x\in \mathcal{V}} (f_x-\overline{f})^2 \eta(x) \leq \kappa \sum_{\substack{x,y\in \mathcal{V},\\ x\sim y}}  \left(f(x)-f(y)\right)^2 
	\end{align*}
	with \begin{align}\label{kap}\kappa:= 0.5  l_{max}(\tilde{\Gamma}) \max_{x\in \mathcal{V}}\eta^2(x) b_{max}(\tilde{\Gamma})\end{align} where $l_{max}(\tilde{\Gamma})$ 
	is the maximal length of all possible paths and \linebreak $b_{max}(\tilde{\Gamma}):=\max_{e\in\mathcal{E}}\ewx{\Gamma}{b(e,\Gamma)}$  is the maximal average load.
\end{remark}

Since we wish to prove a Poincar\'{e} inequality being local w.r.t. to balls 
we will apply the previous principle to a subgraph of $G$ corresponding to the points belonging to  some  ball.

We  define a subgraph $G_B$ of a graph $G=(\mathcal{V},\mathcal{E})$ to be the graph with vertex set $B\subset \mathcal{V}$ and edge set $\mathcal{E}_B:=\{e=(a,b)\in \mathcal{E} \text{~with~} a,b \in B\}\subset \mathcal{E}$. If $\eta$ is a measure on $G$, then the induced measure $\tilde{\eta}$ on $G_B$ is given by the point measure  $\tilde{\eta}(x)=\frac{\eta(x)}{\sum_{x\in B}\eta(x)}=\frac{\eta(x)}{\eta(B)}$ for all $x\in B$ with $\eta(B)>0$. 
\nomenclature[Cz]{$G_A$}{subgraph with vertex set $A$}
\nomenclature[Cz]{$\tilde{\eta}$}{graph measure conditioned on subset}

\begin{cor}\label{cor:gs}
	Let $G=(\mathcal{V}, \mathcal{E})$ be a given graph and $\eta$ a discrete probability measure on $\mathcal{V}$.
	
	Then for a connected subgraph  $G_B=(B,\mathcal{E}_B)$ with $B\subseteq \mathcal{V}$ and $\eta(B)>0$  there exists a quantity $\kappa_B>0$ such that for all functions $f: \mathcal{V}\rightarrow \RZ$ 
	the inequality 
	\begin{align}\label{gs-3}
	\sum_{x\in B} (f_x-\overline{f}_B)^2 \eta(x)  \quad 
	&\leq  \tilde{\kappa}_B 
	\sum_{x\in B}\sum_{\substack{y\in B\\ y\sim x}} (f_x-f_y)^2 
	\end{align}
	where
		$\overline{f}_B=\frac{\sum_{x\in B} f_x \eta(x)}{\sum_{z\in B}\eta(z)}$
	holds. For a set of possible paths $\tilde{\Gamma}_B$ and a randomly chosen set of paths $\Gamma_B$ we obtain
	\begin{align}\label{kapfinal}\tilde{\kappa}_B:= 0.5 \left(\sum_{z\in B}\eta(z)\right)^{-1} l_{max}(\tilde{\Gamma}_B) \max_{x\in B}\eta^2(x) 
	b_{max}(\tilde{\Gamma}_B).\end{align}
\end{cor}

\begin{remark}
	Note that the double sum of the right side of \cref{gs-2,gs-3} corresponds with the Dirichlet form generated by some operator $L$ and associated measure $\eta$:
	\[\sum_{x,y\in B}  \left(f(x)-f(y)\right)^2 a_{xy}=\int_B f(x)L_1 f(x) d\eta_1(x)=:\Xi(f,f)\]  for $L_1=\mathbf{D}-\mathbf{A}$ and
	\[\sum_{x,y\in B}  \left(f(x)-f(y)\right)^2 a_{xy}=\int_B f(x) L_2 f(x) d\eta_2(x)=:\Xi(f,f)\]  for $L_2=\Einh -\mathbf{D}^{-1}\mathbf{A}$.	
\end{remark}
	
In particular, we will use \cref{cor:gs} for a subset of the form $B=\overline{B}_{\mathcal{V},\distb{\manifold{M}}}(x,r)$.
In this case, an upper bound of $\tilde{\kappa}_B$ of the form $c_{\kappa}\cdot r^2$ can be derived under the \cref{ass::poseta,ass::Ahlfors,ass::lip}.
if we choose an adequate set $\tilde{\Gamma}$ of specific paths.

\begin{lem}[$\kappa$ bound]\label{cor::kappabound1}
	Let $G=(\mathcal{V}, \mathcal{E})$ be an $\varepsilon$-graph defined from an i.i.d. sample of size $n$ from  the probability  measure $\mu$ on the submanifold $\manifold{M}$ of $\RZ^K$ such that \cref{ass::manifold,ass::graph,ass::Ahlfors,ass::lip} are satisfied with parameters $\lambda_1, \lambda_2, c_l,c_u, k, L_{min}^*,L_{max}^*$ and $\eta$ the graph measure.\\
	Let's consider $B:=\overline{B}_{\mathcal{V},\distb{\manifold{M}}}(X_i,r_{\manifold{M}})$ with $0<r_{\manifold{M}}<r_{max}=\min \left( i(\manifold{M})/2,\frac{\pi}{4\sqrt{\Lambda}}\right)$ and $X_i\in\mathcal{V}$. %
	Let $\delta\in (0,1)$ and denote $\pr{6}
	:=2 \exp\left(-\frac{\delta^2 (n-1) c_l r_{\manifold{M}}^k}{3}\right)$ and \\$
	\pr{7}:=2	\left(\frac{2 \sqrt{k+3} r_{\manifold{M}}}{L_{min}^*\varepsilon}\right)^k  \exp\left(-\frac{1}{3}\delta^2 (n-1) c_l \frac{\varepsilon^k {L_{min}^*}^k}{4^k \sqrt{k+3}^k {L_{max}^*}^k}\right)$.
	
	If 
	$n\geq \frac{1}{(1-\delta)c_l} \left(\frac{4 \sqrt{k+3} L_{max}^*}{L_{min}^* \varepsilon}\right)^k +1$ 
	and $\frac{\sqrt{k+3}}{L_{min}\varepsilon}\geq 1$, 
	then, with probability at least $1-\pr{7}-\pr{6}$, there exists a constant $C_{\kappa}=C_{\kappa}(n,\varepsilon, k, L_{min}^*,L_{max}^*,c_l,c_u,\delta)>0$ such that 
	
	\begin{align}\label{ineq:k}
	\tilde{\kappa}_B \leq C_{\kappa} r_{\manifold{M}}^2.   
	\end{align}
	The constant is given by \begin{align}
	C_{\kappa}:=    \frac{1}{n \varepsilon^{k+2}}
	\frac{(\eta^+)^2}{(1-\delta)c_l  \eta^-} 	\left(1+w\right)^2 k^2 \left(2\frac{\sqrt{k+3}}{L_{min}^*}\right)^{k+2}
	\end{align}
	with $w:=\frac{2(1+\delta)}{(1-\delta)}\frac{c_u}{c_l}\frac{{L_{max}^*}^k}{{L_{min}^*}^k}4^k \sqrt{k+3}^k$.
\end{lem}
\nomenclature[Cz]{$\mu_{min}, \mu_{max}$}{minimal and maximal measure of a ball (radius ??)} 
\nomenclature[Cz]{$\pr{a}$}{probability, Ahlfors} 
\nomenclature[Cz]{$\pr{bmax}$}{probability, $b_{max}$ }
\nomenclature[Cz]{$C_{\kappa}$}{constant to bound $\kappa$}
\nomenclature[Cz]{$\pr{4}$}{probability to bound $\kappa$}  
The proof can be found in the next section.

Based on \cref{cor:gs} 
and \cref{cor::kappabound1} 
it follows a Poincar\'{e}-type inequality in $\distb{\manifold{M}}$-distance. To transport this to a local Poincar\'{e} inequality w.r.t. the $\distb{SP}$-distance, we need 
the following lemma which states that the variance of a function $f$ over a ball $B_1$  w.r.t.\ the graph measure can be upper bounded by the variance of $f$ over a larger ball $B_2\supset B_1$ times the factor $\eta_{B_2}/\eta_{B_1}$.

\begin{lemma}\label{lem::imp-ineqV1}
	Let $B_1$ and $B_2$ be two sets satisfying $B_1\subset B_2\subset \mathcal{V}(G)$ and $\eta$ defined on $\mathcal{V}(G)$. 
	Then the inequality 
	\begin{align}\label{imp-ineq}
	\sum_{x\in B_1} (f(x)-\overline{f}_{B_1})^2\eta(x) \leq \sum_{x\in B_2} (f(x)-\overline{f}_{B_2})^2\eta(x)
	\end{align}
	holds with  $\overline{f}_{B_i}=\frac{\sum_{x\in B_i} f(x) \eta(x)}{\sum_{y\in B_i} \eta_y }$ for $i\in \{1,2\}$. 
\end{lemma}
\begin{proof}[of Lemma \ref{lem::imp-ineqV1}] 
	Observe that $\overline{f}_{B_1}=\argmin_c \sum_{x \in B_1} (f(x)-c)^2 \eta(x)$. Therefore, for any $c\in \RZ$ and especially for $c=\overline{f}_{B_2}$, we have
	\begin{align*}
	\sum_{x \in B_1} (f(x)-\overline{f}_{B_1})^2 \eta(x) &\leq  \sum_{x \in B_1} (f(x)-c)^2 \eta(x)\\
	&  \leq \sum_{x \in B_2} (f(x)-c)^2 \eta(x)
	\end{align*}
	since the summands are nonnegative.
\end{proof}

\begin{proof}[of \cref{theo:LPIsp} for one given ball] 
	Let $r\in \NZ, r\geq 1$.\\
	Denote  $B_1$:=$\overline{B}_{SP}(X_i,r)$, $B_2:=\cB{X_i}{r_{\manifold{M}} } {\mathcal{V},\distb{\manifold{M}}}$  with $ r_{\manifold{M}} =\varepsilon (1-\lambda_1)^{-1} r$ and\\ $B_3:=\cB{X_i}{\left(4 \frac{(1+\lambda_2)}{(1-\lambda_1) }+1\right) r}{SP}$.
	
	We restrict our computations to the high-probability event where the inequalities \cref{ineq::dist,ineq:k} 
	hold simultaneously. This event occurs with probability at least $1-\pr{1}-\pr{7}-\pr{6}$.
	
	Then by  \cref{cor::dist} and since $r\geq 1$ 
	these sets satisfy  $B_1\subseteq B_2\subseteq B_3$. 

	We can write 
	\begin{align*}
	\sum_{x\in B_1} (f(x)-\overline{f}_{B_1})^2\eta(x) &\leq \sum_{x\in B_2} (f(x)-\overline{f}_{B_2})^2 
	~~\text{~  by Lemma \ref{lem::imp-ineqV1}}\\
	&\leq  \tilde{\kappa}_{B_2} \frac{1}{2} \sum_{x\in B_2}\sum_{y\in B_2, y\sim x} (f_x-f_y)^2 
	\text{by \cref{cor:gs}}\\
	&\leq  C_{\kappa}~ r_M^2 \sum_{x,y\in B_2, x\sim y} (f(x)-f(y))^2 
	~~ \text{~by \cref{cor::kappabound1}}\\
	& \leq    C_{\kappa}  ~r_M^2   \sum_{x,y\in B_3, x\sim y} (f(x)-f(y))^2
	~~\text{~ $B_2\subseteq B_3$}\\
	& \leq C_{\kappa} ~ \frac{\varepsilon^2 }{(1-\lambda_1)^2}  r^2 \sum_{x,y\in B_3, x\sim y} (f(x)-f(y))^2
	~~ \text{~~def. of } r_{\manifold{M}}
	\end{align*}
	We applied first \cref{lem::imp-ineqV1} using $B_1$ and $B_2$ and then  \cref{cor:gs} and \cref{cor::kappabound1} for $B_2$. Finally we used that the sum of nonnegative summands increases when increasing the number of summands by replacing $B_2$ with 
	$B_3$. 
\end{proof}
It remains to show that the proven inequality holds true simultaneously for all balls with high probability.

\begin{remark}\label{rem:dist}
	The distance approximation in \cref{cor::dist} holds uniformly for all points $x,y\in \mathcal{V}$.
	This implies that the ball inclusions hold uniformly for all center points $X_i\in \mathcal{V}$ and all radii.
\end{remark}

\begin{remark}\label{rem:kap}
	We observe that in \cref{cor::kappabound1} the probabilities $\pr{6}$ and $\pr{7}$ depend on the radius but not on the center point of the considered ball.
	We can deduce the following uniform result based on the uniform bound:
	Let $\mathcal{R}_{\manifold{M}}$ be a finite set of radii $r_{\manifold{M}}$ with $\abs{\mathcal{R}_{\manifold{M}}}\leq n$ and 
	$\mathcal{V}$ the finite set of center points $x$ ($\abs{\mathcal{V}}=n$).
	Then with probability at least
	\begin{align*}&1- 2 n^2  	\left(\frac{2 \sqrt{k+3} \max_{r_{\manifold{M}}\in \mathcal{R}_{\manifold{M}}} r_{\manifold{M}}}{L_{min}^*\varepsilon}\right)^k  \exp\left(-\frac{1}{3}\delta^2 (n-1) c_l \frac{\varepsilon^k {L_{min}^*}^k}{4^k \sqrt{k+3}^k {L_{max}^*}^k}\right)\\
	& -  2n^2 \exp\left(-\frac{\delta^2 (n-1) c_l \min_{r_{\manifold{M}}\in \mathcal{R}_{\manifold{M}}}r_{\manifold{M}}^k}{3}\right)\end{align*}
	the inequality
	\[\tilde{\kappa}_B \leq C_{\kappa} r_{\manifold{M}}^2\]
	holds for all (finite many) balls $B=\cB{X_i}{r_{\manifold{M}}}{\manifold{M}}$ with $r_{\manifold{M}}\in \mathcal{R}_{\manifold{M}}, X_i\in \mathcal{V}$.
\end{remark}

\begin{proof}[of \cref{theo:LPIsp} uniformly for all balls]\mbox{}\\
	Under $r_{max}\varepsilon^{-1}(1-\lambda_1)\geq 1$ we have $\mathcal{R}:=\NZ\cap \left(0,\min(n, r_{max}(1-\lambda_1)\varepsilon^{-1})\right)\neq \emptyset$.
	As a consequence of the \cref{rem:dist,rem:kap} applied to the radii  $r_{\manifold{M}} :=\varepsilon (1-\lambda_1)^{-1} r$ for $r\in \mathcal{R}$ the local Poincaré inequality \cref{lpi} holds uniformly for all balls $\overline{B}_{SP(x,r)}$ with $r\in \mathcal{R}$ and $X_i\in \mathcal{V}$ with probability at least 
	\begin{align*}&1-2n^2	\left(\frac{2 \sqrt{k+3} \max_{r_{\manifold{M}}\in \mathcal{R}_{\manifold{M}}} r_{\manifold{M}}}{L_{min}^*\varepsilon}\right)^k  \exp\left(-\frac{1}{3}\delta^2(n-1) c_l \frac{\varepsilon^k {L_{min}^*}^k}{4^k \sqrt{k+3}^k {L_{max}^*}^k}\right) \\
	&-  2n^2 \exp\left(-\frac{\delta^2 (n-1) c_l \min_{r_{\manifold{M}}\in \mathcal{R}_{\manifold{M}}}r_{\manifold{M}}^k}{3}\right)-\pr{1}.\end{align*}	
	Observe that
	$\max_{r_{\manifold{M}}\in \mathcal{R}_{\manifold{M}}} r_{\manifold{M}}= n \varepsilon (1-\lambda_1)$ and $\min_{r_{\manifold{M}}\in \mathcal{R}_{\manifold{M}}}r_{\manifold{M}}^k=\varepsilon^k(1-\lambda)^{-k}$
	$r_{max}(1-\lambda_1)\varepsilon^{-1}$
	since $1\leq r\leq n$ for $r\in \mathcal{R}$. Finally substitute $n-1$ by $n/2$ in the expression.
	
	The extension to non-integer radii is still missing. Let $r\in \mathcal{R}$	and $\tilde{r}\in[r,r+1)$. Then  $\lfloor \tilde{r}\rfloor=r$ and 	$\cB{x}{\tilde{r}}{SP}=\cB{x}{r}{SP}$. Applying the LPI for $r$ and substituting $r$  by $\tilde{r}$ on the right side (within the factor and as ball radius; this is possible since $\tilde{r}\geq r$) we obtain the wanted inequality for the non-integer radius $\tilde{r}$.
\end{proof}

\subsection{Proof of  \cref{cor::kappabound1}}
\label{sec:kappa}
Now we will prove the bound for the constant $\tilde{\kappa}_B$ introduced in \cref{cor::kappabound1}. 
The aim is to find an upper bound of order $\bigO(r_{\manifold{M}}^2)$  with $B:=\overline{B}_{\mathcal{V},\distb{\manifold{M}}}(X_i,r_{\manifold{M}})$ for a suitable range of $r_{\manifold{M}}$. 
From \cref{kapfinal} it follows that the choice of the set  $\tilde{\Gamma}_B$ of possible paths 
is essential for the upper bound of $\kappa_B$. 
For example, if we choose the paths with minimal number of edges and $B=\overline{B}_ {SP}(r)$ then the maximal length is bounded by $2r$. But we cannot control 
$b_{max}$. 

Following the ideas of \citeauthor{Ulrike2014} our strategy is to choose the class of random Hamming paths introduced in \cite{Ulrike2014} as set  $\tilde{\Gamma}$ of possible paths. For this specific class we can bound the path length and maximal average load in an adequate way. 

In contrast to \cite{Ulrike2014}, the vertices of the graph are drawn from a submanifold and we consider the restriction of the graph to a ball of radius $r_{\manifold{M}}$.
Moreover, we must precisely keep track of the constants, especially the radius, the sample size $n$ and $\varepsilon$.

We recall the construction of random Hamming paths for a geometric graph $G=(\mathcal{V}, \mathcal{E})$  with vertices in the unit cube $[0,1]^k$ which is based on deterministic Hamming paths between cells. 

Let's consider a regular grid on the cube with grid width $g$ (such that $1/g\in \NZ$) and assume that the following is satisfied:

\begin{enumerate}[label=\roman*)]
	\item \label{onepoint} each grid cell contains at least one point of $\mathcal{V}$
	and \item \label{connec} points in the same and in neighboring grid cells are connected in the graph.
\end{enumerate}

The Hamming cell path from cell $A$ to cell $B$ is 
the shortest sequence of adjacent grid cells such that
in the first segment of the path the cells differ only in the first coordinate of their center points, in the second segment of the path the cells differ only in the second coordinate of their center points, and so on.

A random Hamming path between two vertices $x,y\in \mathcal{V}$ is constructed
in the following way. We take the Hamming cell path between the cells containing $x$
and $y$ and then choose randomly one point $z_i$ in each of interior cells of this Hamming cell path. By \ref{onepoint} the points $z_i$ exists and by \ref{connec} the chosen points in neighboring cells are connected by an edge in the graph. Therefore  the random sequence $x=z_0,z_1,\ldots,z_{l-1},y=z_l$ of points determines a path in the graph, a so-called random Hamming path. For $x$ and $y$ in the same cell or in neighboring cells, just take the edge $e=(x,y)$ as (random) Hamming path. Then the following is known (\cite{Ulrike2014}).

\begin{lemma}\label{cube}
	Let's consider an $\varepsilon$-graph $G=(\mathcal{V},\mathcal{E})$ with vertices in the unit cube $[0,1]^k$, a regular grid on $[0,1]^k$ of grid with $g$  and $\tilde{\Gamma}$ the set of all possible random Hamming paths on the graph.
	Let $N_{min}$ and $N_{max}$ be the  minimal and  maximal number of points per grid cell.
	If $N_{min}\geq 1$ and $g \leq \frac{\varepsilon}{\sqrt{k+3}}$,
	then
	\begin{align*}
	l_{max} &\leq k\cdot \frac{1}{g}\\
	\text{and~}~	b_{max} &\leq
	\left(1+\frac{N_{max}}{N_{min}}\right)^2 
	\frac{k}{g^{k+1}}.
	\end{align*}
\end{lemma}

\begin{proof}
	Observe that the Euclidean distance of two points in neighboring cells is at most $g\cdot \sqrt{k+3}$. Therefore $g \leq \frac{\varepsilon}{\sqrt{k+3}}$ implies the assumption \cref{connec}
	Moreover, $N_{min}\geq 1$ implies the assumption \cref{onepoint}. Thus  random Hamming paths for all pairs of points $x,y\in \mathcal{V}$ exist. The proofs of the bounds can be found in \cite{Ulrike2014} (see proof of Theorem 6 and proof of Proposition 22). 
	In \cite{Ulrike2014} they obtain \[b_{max}\leq 1+\left(\frac{N_{max}^2}{N_{min^2}}+2\frac{N_{max}}{N_{min}}\right) \frac{k}{g^{k+1}}.\] 
	Using $1\leq  \frac{k}{g^{k+1}}$ we get
	\[b_{max}\leq \left( 1+\frac{N_{max}^2}{N_{min^2}}+2\frac{N_{max}}{N_{min}} \right) \frac{k}{g^{k+1}}=\left(1+\frac{N_{max}}{N_{min}}\right)^2\frac{k}{g^{k+1}}.\]
\end{proof}

In the general case where we consider $\mathcal{V}\subset \manifold{M}$ or even more specifically  $\mathcal{V}\subseteq \overline{B}_{\manifold{M}}(x,r_{\manifold{M}})$ we can get back to the previous situation by mapping back the points by a bi-Lipschitz homeomorphism as defined in \cref{bilipDef}. 

This bi-Lipschitz homeomorphism enables us (as introduced by \cite{Ulrike2014}) to generate paths on the graph $G=(\mathcal{V},\mathcal{E})$ with $\mathcal{V}\subset \mathcal{X}$ based on random Hamming paths in the cube with known properties. The image of the graph under $h$ is a graph $\hat{G}$ with vertex set $h(\mathcal{V})$ and two points $x,y\in h(\mathcal{V})$ are connected in $\hat{G}$ whenever their preimages are connected in $G$. Considering a regular grid on $[0,1]^k$ with grid width $g$ such that \cref{onepoint,connec} are satisfied we can construct random Hamming paths on $\hat{G}$ for every two points $x,y\in h(\mathcal{V})$. Each of these paths corresponds to a path in $G$ by mapping the points of the path back to $\mathcal{V}$ via $h^{-1}$ where they are still connected. These paths on $G$ exhibit by construction the same properties $l_{max}$ and $b_{max}$ as the random Hamming paths on $\hat{G}$.

\begin{cor}[bounds on $l_{max}$ and $b_{max}$]\label{lmax}
	Let $G=(\mathcal{V},\mathcal{E})$ be an $\varepsilon$-graph defined from an i.i.d. sample of size $n$ from the probability measure $\mu$ on the submanifold $\manifold{M}$ of $\RZ^K$ such that \cref{ass::manifold,ass::graph,ass::lip} 
	are satisfied with parameters $n,\varepsilon, k, L_{min}^*, L_{max}^*$.
	Let  $X_i\in \mathcal{V}$ 
	Let $\mathcal{X}:=\cB{X_i}{r_{\manifold{M}}}{\manifold{M}}$.
	Let $\delta\in(0,1)$.
	Denote
	$w_-:=\min_{x\in \manifold{M}}\mu(B_{\manifold{M}}(x, \frac{L_{min}\varepsilon}{4\sqrt{k+3}L_{max}}))$, $\tilde{w}_+=\max_{x\in \manifold{M}}\mu(B_{\manifold{M}}(x, \varepsilon))$ and  $\pr{8}:=2	\left(\frac{2 \sqrt{k+3}}{L_{min}\varepsilon}\right)^k \exp\left(-\frac{1}{3}\delta^2(n-1) w_-\right)$.
		
	If 
	$n\geq \frac{1}{(1-\delta) w_-} +1$ and $\frac{\sqrt{d+3}}{L_{min}\varepsilon}\geq 1$ then with probability 
	$1-\pr{8}$ there exists a class of paths $\tilde{\Gamma}$ on the subgraph $G_B$ with vertex set $B=\mathcal{V}\cap \mathcal{X}$ such that 
	\begin{align}
	l_{max}(\tilde{\Gamma})
	&\leq 2k\frac{\sqrt{k+3}}{L_{min}\varepsilon}\label{lmaxeq}\\
	b_{max}(\tilde{\Gamma})
	&\leq\left(1+\frac{2(1+\delta)\tilde{w}_+}{(1-\delta)w_-}\right)^2k \left(2\frac{\sqrt{k+3}}{L_{min}\varepsilon}\right)^{k+1} \label{bmaxeq}
	\end{align} 
	are satisfied. 
\end{cor}

\begin{proof}
	We follow the ideas of \cite{Ulrike2014} to prove the corollary.
	As described above, we map the graph to the cube using the bi-Lipschitz homeomorphism $h$, construct random Hamming paths in the cube and map them back.
	Let $h(B)$ the image of $B$ under $h$.
	We set $1/g :=\left\lceil  \frac{\sqrt{k+3}}{L_{min}\varepsilon}\right\rceil$. Then obviously $1/g\in\NZ$ 
	is satisfied. 
	
	Now take two points $x,y\in h(B)$ in neighboring cells of the cube, then since $h$ is bi-Lipschitz and due to the definition of $g$ we get that the Euclidean distance of the preimages of $x,y$ is smaller than $\varepsilon$:
	\[\norm{h^{-1}(x)-h^{-1}(y)} \leq \norm{h^{-1}(x)-h^{-1}(y)}_{\manifold{M}}\leq \frac{1}{L_{min}}\norm{x-y} \leq \frac{1}{L_{min}} g \sqrt{k+3} \leq \varepsilon. \]
	Thus, \cref{connec} holds. 
	It remains to show that \cref{onepoint} holds. 
	Let $C_i$ be a grid cell of the cube with center point $c_i$
	and define  balls $B_i^1,B_i^2$ in the submanifold centered at $h^{-1}(c_i)$ with radii $r:=\frac{L_{min}\varepsilon}{4\sqrt{k+3}L_{max}} \leq \frac{g}{2L_{max}}$ and $R:=\varepsilon>\frac{\sqrt{k}g}{L_{min}}$.
	Conditionally on $X_i$,  these balls are deterministic and they  include $X_i$. That implies that $h(X_i)$ belongs to one grid cell. The number of points in a grid cell is determined by the remaining $n-1$ vertices (possibly increased by one, if $h(X_i)$ belongs to that grid cell.) 
	As shown in \cite{Ulrike2014} we obtain $B_i^1 \subseteq h^{-1}(C_i)\subseteq B_i^2$. We also have $B_i^1\subseteq \mathcal{X}$.
	Then the number of points in $C_i$ is bounded from below by the number of points in $B_i^1$ denoted by $\tilde{N}_i^1$ ($X_i$ is excluded). Let $\tilde{N}_{min}^1:=\min_{i=1,\ldots,W} \tilde{N}_i^1$
	Applying \cref{cor:counting} for these $W:=1/(g^k)$ balls $B_i^1$ we get
	with $w_-\leq \min_{i=1,\ldots,W}\mu(B_i^1)$
	\begin{align*}
	\prob{\tilde{N}_{min}^1\leq (1-\delta) (n-1) w_-}&\leq W \exp\left(-\frac{1}{3}\delta^2 (n-1) w_-\right) \\ &\leq \left(\frac{2 \sqrt{k+3}}{L_{min}\varepsilon}\right)^k \exp\left(-\frac{1}{3}\delta^2 (n-1) w_-\right).
	\end{align*}
	Thus  
	\cref{onepoint} holds with probability at least	$1-	W \exp\left(-\frac{1}{3}\delta^2  (n-1) w_-\right)$ if $n\geq \frac{1}{(1-\delta) w_-} +1$ is satisfied. 
	
	That means we can construct random Hamming paths for every pair of points $x,y\in  h(B)$ and we can apply \cref{cube}.
	\cref{lmaxeq} is then obvious. 
	Let $\tilde{N}_{i}^2$ be the random number of points in $B_i^2$
	and $\tilde{N}_{max}^2:=\max_{i=1\ldots,W} \tilde{N}_i^2$.
	Taking $X_i$ into account the maximal number of points in a grid cell is upper bounded by $\tilde{N}_{max}+1$.  
	Analogous to the computation above 
	we get with $\tilde{w}_-:=\min_{x\in \mathcal{X}}  \mu(B_{\manifold{M}}(x,\varepsilon)) \leq \min_{i=1,\ldots,l} \mu(B_i^2)$ that
	\begin{align*}
	\prob{\tilde{N}_{max}\geq (1+\delta) (n-1)\tilde{w}_+}&\leq W \exp\left(-\frac{1}{3}\delta^2 (n-1) \tilde{w}_-\right) \\&\leq \left(\frac{2 \sqrt{k+3}}{L_{min}\varepsilon}\right)^k \exp\left(-\frac{1}{3}\delta^2 (n-1) w_-\right).
	\end{align*}
	Observe that the results of the computations conditionally on $X_i$ do not depend on $X_i$, so they hold unconditionally.
	Thus 
	\[\frac{\tilde{N}_{max} +1}{\tilde{N}_{min}} \leq \frac{2(1+\delta) (n-1)\tilde{w}_+}{(1-\delta) (n-1) w_-}=\frac{2(1+\delta)\tilde{w}_+}{(1-\delta)w_-}\] holds with probability $1-2W \exp\left(-\frac{1}{3}\delta^2 (n-1) w_-\right)$.
	Inserting this quantity and $g$ in \cref{cube} finishes the proof.
\end{proof}
\nomenclature[cz]{$\mu_{min}$}{$\min_{x_i} \mu(B(h^{-1}(c_i),\varepsilon\frac{L_{min}}{2 \sqrt{d+3}L_{max}}))$}
\nomenclature[Cz]{$\mu_{max}$}{$\max_{x_i} \mu(B(h^{-1}(c_i),\varepsilon))$}
\nomenclature[Cz]{$g$}{grid width}

Now we can prove \cref{cor::kappabound1}.
\begin{proof}
	Let's consider $\mathcal{X}:=\overline{B}_{\manifold{M}}(x,r_{\manifold{M}})$ and $B:=\overline{B}_{\mathcal{V},\distb{\manifold{M}}}(x,r_{\manifold{M}})$. Under the assumptions of \cref{lmax} and \cref{ass::lip} we get from \cref{kapfinal,lmaxeq,bmaxeq} that
	\begin{align}\tilde{\kappa}_B
	&\leq 0.5 \frac{\max_{x\in B}\eta^2(x) }{\sum_{z\in B}\eta(z)}	\left(1+\frac{2(1+\delta)\tilde{w}_+}{(1-\delta)w_-}\right)^2 k^2 \left(2\frac{\sqrt{k+3}r_{\manifold{M}}}{L_{min}^*\varepsilon}\right)^{k+2}
	\end{align}
	with probability at least $1-\pr{8}$.
	Now we bound further the quantity under the Ahlfors assumption \labelcref{ass::Ahlfors}. 
	Note that
	\begin{align}
	\left(\sum_{z\in B}\eta(z)\right)^{-1} &\leq \frac{1}{n_B \min_{x\in B }\eta(x)}. 
	\end{align}
	Under the Ahlfors assumption we have
	\begin{align}
	\frac{2(1+\delta)\tilde{w}_+}{(1-\delta)w_-} &\leq 2\frac{1+\delta}{1-\delta}\frac{c_u \varepsilon^k L_{max}^k 4^k \sqrt{k+3}^k }{c_l \varepsilon^k L_{min}^k }\\
	&\leq 2 \frac {1+\delta}{1-\delta}\frac{c_u}{c_l}\frac{L_{max}^k}{L_{min}^k}4^k \sqrt{k+3}^k=:w,\\
	\pr{8}&\leq 2	\left(\frac{2 \sqrt{k+3}}{L_{min}\varepsilon}\right)^k  \exp\left(-\frac{1}{3}\delta^2 (n-1) c_l \frac{\varepsilon^k L_{min}^k}{4^k \sqrt{k+3}^k L_{max}^k}\right)=:\pr{7}
	\end{align}
	and by  \cref{theo::ahlf} \cref{part3}
	with probability at least $1-\exp\left(-\frac{\delta^2 (n-1) c_l r^k}{3}\right)$
	\begin{align}
	n_B\geq (n-1) (1-\delta)c_l r_{\manifold{M}}^k.  
	\end{align}
	
	Then, by plugging in these bounds, we obtain that with probability at least \linebreak $1-\pr{7}-\exp\left(-\frac{\delta^2 (n-1) c_l r^k}{3}\right)$ 
	\begin{align}
	\tilde{\kappa}_B	&\leq \frac{\max_{y\in B}\eta^2(y)}{n(1-\delta)c_l  \min_{x\in B}\eta(x)} 	\left(1+w\right)^2 k^2 \left(2\frac{\sqrt{k+3}}{L_{min}^*\varepsilon}\right)^{k+2}r_{\manifold{M}}^2\\
	&\leq \frac{(\eta^+)^2}{n(1-\delta)c_l  \eta^-} 	\left(1+w\right)^2   k^2 \left(2\frac{\sqrt{k+3}}{L_{min}^*\varepsilon}\right)^{k+2}r_{\manifold{M}}^2.
	\end{align}
	holds.
\end{proof}

\printbibliography[heading=bibintoc,title={References}]

\end{document}